\documentclass[twocolumn]{autart}
\usepackage[bookmarks=false]{hyperref}
\usepackage{graphicx}
\usepackage{subcaption}
\usepackage{amsmath}
\usepackage{amssymb}
\usepackage{mathtools}
\usepackage{comment}
\usepackage{dsfont}
\usepackage{graphicx}
\usepackage{xcolor}
\newenvironment{remark}{
  \par\smallskip
  \noindent\textbf{Remark.}\ \normalfont
}{
  \par\smallskip
}
\newtheorem{theorem}{Theorem}
\newtheorem{lemma}{Lemma}
\newtheorem{proposition}{Proposition}
\theoremstyle{remark}
\newtheorem{assumption}{Assumption}
\newtheorem{problem}{Problem}
\begin{document}
\newcommand{\R}{\mathbb{R}}
\newcommand{\I}{\mathcal{I}}
\newcommand{\J}{\mathcal{J}}
\newcommand{\0}{\mathbf{0}}
\newcommand{\mat}{\mathrm{mat}}
\newcommand{\tr}{\mathrm{tr}}
\newcommand{\col}{\mathrm{col}}
\newcommand{\blkdiag}{\operatorname{blkdiag}}
\newcommand{\BP}{\noindent{\bf Proof. }}
\newcommand{\EP}{\hspace*{\fill} $\blacksquare$\smallskip\noindent}
\newcommand{\BPof}[1]{\noindent\textbf{Proof of #1. }}
\setlength{\abovedisplayskip}{10pt plus 2pt minus 2pt}
\setlength{\belowdisplayskip}{10pt plus 2pt minus 2pt}
\setlength{\abovedisplayshortskip}{2pt plus 2pt}
\setlength{\belowdisplayshortskip}{3pt plus 2pt minus 2pt}
\overfullrule=5pt

\begin{frontmatter}

\title{Distributed Stability Certification and Control\\ from Local Data\thanksref{footnoteinfo}}

\thanks[footnoteinfo]{The work of authors is supported by NWO grant OCENW.KLEIN.257.}

\author[Address]{Surya Malladi}\ead{v.s.p.malladi@rug.nl},
\author[Address]{Nima Monshizadeh}\ead{n.monshizadeh@rug.nl}

\address[Address]{ENTEG, University of Groningen}  
          

\begin{abstract}
Most data-driven analysis and control methods rely on centralized access to system measurements. In contrast, we consider a setting in which the measurements are distributed across multiple agents and raw data are not shared. Each agent has access only to locally held samples, possibly as little as a single measurement, and agents exchange only locally computed signals. Consequently, no individual agent possesses sufficient information to identify the entire system or synthesize a controller independently. To address this limitation,
we develop distributed dynamical algorithms that enable the agents to collectively compute global system certificates from local data. Two problems are addressed. First, for stable linear time-invariant (LTI) systems, the agents compute a Lyapunov certificate by solving the Lyapunov equation in a fully distributed manner. Second, for general LTI systems, they compute the stabilizing solution of the algebraic Riccati equation and hence the optimal linear-quadratic regulator (LQR). An initially proposed scheme guarantees practical convergence, while a subsequent augmented PI-type algorithm achieves exact convergence to the desired solution. We further establish robustness of the resulting LQR controller to uncertainty and measurement noise. The approach is illustrated through distributed Lyapunov certification of a quadruple-tank process and distributed LQR design for helicopter dynamics.
\end{abstract}
\end{frontmatter}

\section{Introduction}
Data-driven control has emerged as a powerful alternative to classical model-based approaches, offering the ability to synthesize controllers directly from collected data. In contrast to conventional methods, which rely on first identifying a parametric model of the system and subsequently designing a controller, data-driven methods seek to design control policies directly from measured input-state/output behavior. {This paradigm is particularly attractive for complex or poorly modeled systems, where accurate system identification is difficult, costly, or infeasible.}

Within the large body of literature on data-driven methods, 
we mention 
virtual reference feedback tuning \cite{campi2002virtual}, data-enabled predictive control (DeePC) \cite{coulson2019data}, policy gradient methods \cite{hu2023toward,zhao2023data}, and data-based semidefinite programs \cite{de2019formulas,van2020data,markovsky2021behavioral}.  Such methods provide a streamlined framework for controller synthesis, and have been extended to cover stochastic, {time-varying,} nonlinear, and robust settings \cite{bianchin2023online,nortmann2023direct,dai2023data,berberich2020data}. However, a common and often implicit assumption in nearly all these methods is that the data is centrally available stored in one place and entirely accessible by the designer. While centralization simplifies analysis and control design, its limitations become apparent in certain modern scenarios, as outlined below.

\textit{Motivation and Challenges.} 
Modern engineered systems generate large volumes of data that are collected and stored across multiple locations rather than in a single repository. This non-centrality arises because measurements may stem from different experimental runs, be owned by distinct agents, or be distributed across subsystems of a large-scale network. Operational, organizational, and regulatory constraints often prevent pooling such data, particularly in the absence of a trusted coordinating entity, as exemplified by emerging data markets \cite{zheng2022fl,baghcheband2024machine}. Similar restrictions appear in privacy-preserving paradigms such as federated learning, where raw measurements remain local and only processed information is exchanged. Even when technically feasible, centralization can be impractical due to the substantial storage and memory demands associated with aggregating large datasets at a single location. Moreover, centralized repositories pose privacy and security risks, as measurements may reveal sensitive structural and behavioral information that can be exploited to launch data-driven attacks \cite{li2019optimal,gao2023data} or to identify models enabling model-based attacks \cite{teixeira2015secure,pasqualetti2013attack}.

For control systems, this lack of centralized access poses a fundamental difficulty. Existing data-driven analysis and controller synthesis methods implicitly assume that all available measurements can be accessed and processed centrally. When measurements are distributed across agents and raw data are not shared, no single entity has sufficient information to identify the system or compute a data-driven controller independently. These considerations motivate control synthesis methods that operate without centralized data access. In this work, we develop distributed solutions to data-driven control problems relying only on locally available information.
In particular, we consider a setting in which measurements are held across multiple agents and raw data are not shared. Each agent has access only to a subset of data samples, a single sample in the extreme case, and communicates with neighboring agents through locally computed variables. Due to the lack of sufficient information for any individual agent, the objective is therefore to perform analysis and control tasks collectively from locally held data via distributed interaction. 

\textit{Related work.}   
A few interesting directions in distributed data-driven control share surface similarities with our problem but rely on assumptions that differ fundamentally from the setting addressed here.
First, methods based on data-driven LMI conditions (e.g., \cite{eising2022informativity})
assume that complete input–state trajectories are centrally available, and distribution enters only through structural constraints on the controller rather than through decentralized access to data.
Second, distributed MPC and System-Level Synthesis approaches (e.g., \cite{alonso2022data}) require each subsystem to possess long local trajectories satisfying persistence-of-excitation conditions and typically involve sharing trajectory data or derived local responses during optimization, which presupposes significantly richer local information than is available in our setting. Similarly, the design of distributed controllers within a differential games framework is studied in \cite{cappello2021distributed}, where each subsystem is assumed to have access to its own complete state trajectories as well as those of its neighboring subsystems.
Third, recent distributed data-driven optimal control schemes (e.g., \cite{celi2023distributed}) also use full-length experimental trajectories per agent and employ iterative projection where intermediate high-dimensional variables are exchanged, again assuming data availability and communication patterns incompatible with sample-wise fragmented data. In contrast, the setting studied in this paper assumes that each agent holds only minimal local information, exemplified by a single input–state sample and that raw data cannot be shared, requiring a fundamentally different approach to the aforementioned data-driven control works.

Our analysis is also related to blended dynamics theory, previously used in model-based distributed control, e.g., decentralized synchronization \cite{lee2020tool} and distributed stabilization \cite{kim2023decentralized}. Here we extend its use to a data-driven setting in which the system model is unknown and only locally held measurements are available.

\textit{Contributions.} {We consider a linear time-invariant  system for which the available measurements are distributed across multiple agents. Each agent has access only to limited locally held samples and raw measurements are not shared.
The contributions of this paper are threefold:

First, we provide a distributed data-driven method in which each agent locally computes a low-rank share of the unknown system matrix using only local data and neighbor communication.
    
Second, the agents compute a quadratic Lyapunov function serving as a stability certificate for the system. We propose two distributed dynamical algorithms: the first guarantees practical convergence, while the second incorporates a PI-type augmentation that achieves exact convergence to the Lyapunov solution. Both algorithms are based on a distributed dynamical formulation of the Lyapunov equation, in which each agent’s local share enters its update dynamics.
    
Third, the agents collectively compute the linear-quadratic regulator (LQR) for the system. We develop two distributed dynamical algorithms for this purpose: one achieves practical convergence, while a PI-augmented variant attains exact convergence. The algorithms are based on a distributed (nonlinear) differential Riccati equations, in which each agent’s local share enters the update dynamics and agents exchange only locally constructed signals with their neighbors. We further establish robustness of the resulting controller to model uncertainty and measurement noise.

Together, the proposed methods enable stability certification and optimal LQR synthesis from locally held samples without centralized data access.

{The remainder of the paper is organized as follows. Section \ref{sec: Problem statement} formulates the problems of interest. Section \ref{sec: Data distribution scheme and partial system identification} introduces the data distribution scheme and the distributed computation of low-rank shares of the system matrix. Section \ref{sec: Dist sol to LE} presents distributed algorithms for solving the Lyapunov equation from locally held measurements and neighboring communication, while Section \ref{sec: Dist sol to ARE} develops the corresponding algorithms for computing the LQR controller. Robustness to uncertainty and measurement noise is analyzed in Section \ref{sec: Robustness}. Section \ref{sec: Numerical simulation} provides illustrative case studies, including distributed Lyapunov certification for a quadruple-tank process and distributed LQR design for helicopter dynamics. Concluding remarks are given in Section \ref{sec:concl}. Auxiliary results and some technical proofs are deferred to the Appendix.}

\section{Notations and preliminaries}
The $N$-dimensional vector of all ones is denoted by 
$\mathds{1}_N$. The minimum and maximum singular values of a matrix $A$ are denoted by $\sigma_{\rm min}(A)$ and $\sigma_{\rm max}(A)$, respectively. The $2-$norm of $A$ is denoted by $||A||$. 
The notation $\mathrm{vec}(\cdot)$ is used to denote the vectorization of a matrix. Its inverse operation i.e., matricization to a $n \times m$ matrix of a vector is denoted by $\mat_{n,m}(\cdot)$, with the indices dropped for square matrices. We use $``*$'' to denote the symmetric completion of a block matrix. 
We denote the Kronecker product of two matrices $A$ and $B$ by $A \otimes B$. For matrices $A,B,C$ of appropriate dimensions, the following property holds
$
\mathrm{vec}(ABC)=(C^T \otimes A)\mathrm{vec}(B).
$
The trace of a matrix denoted by $\tr(\cdot)$ satisfies the following property:
\begin{equation}\label{eqn: Cyclic property of trace}
    \tr(ABC)=\tr(BCA)=\tr(CAB).
\end{equation}
The $H_2$ norm of a linear time-invariant system associated with the input-state-output matrices $(A, B, C)$
is given by \cite[Chapter 4]{khalil1996robust}
\begin{equation}\label{eqn: H2 norm}
    ||G||_2=(\tr(B^TW_oB))^{1/2}=(\tr(CW_cC^T))^{1/2},
\end{equation}
where $G(s):=C(sI-A)^{-1}B$, and $W_o$ and $W_c$ 
are, respectively, the observability and controllability Gramians that can be obtained as the solutions to the Lyapunov equations
\[
 A^T W_o + W_o A + C^T C = 0, \; A W_c + W_c A^T + B B^T = 0.
\]
\section{Problem statement}\label{sec: Problem statement}
We investigate two notable data-driven analysis and control problems. The first problem considers a stable linear system and finds a Lyapunov certificate in a distributed fashion. In the second problem, the system is potentially unstable, and optimal stabilizing controllers are distributedly designed.
We emphasize the key feature of this work that neither storage of the data nor computation over the data requires a central location or a central processor. 

To formalize the first problem, consider a continuous-time linear time-invariant (LTI) system with no input
\begin{equation}\label{eqn: LTIH}
    \dot{x}=Ax,
\end{equation}
where the state matrix $A\in \R^{n \times n}$ is unknown. 
The data samples $\{x(t_i), r(t_i)\}$ consist of 
measured state data and rate of change of state data at time $t_i$, with $i \in \I:=\{1,2,\ldots,N\}$.
As mentioned, we work under a {\em fragmented data} regime where data is not centrally available, but spread across $\hat N$ local databases $\mathcal{D}_j$, each accessible to a computing agent $j\in \mathcal{J}:=\{1, \ldots, \hat N\}$.  Without loss of generality, we discuss the result in the maximally fragmented scenario where each agent has only access to a single data sample, i.e. $\mathcal{D}_i=
\{x(t_i), r(t_i)\}$  and $\I=\mathcal{J}$.

We aim to solve the Lyapunov equation for system \eqref{eqn: LTIH}:
\begin{equation}\label{eqn: LE}
    A^TP+PA+Q=0,
\end{equation}
for a given $Q>0$.
\begin{problem}\label{prb: Lyapunov function for a stable system}
Let $A$ be Hurwitz. Given the matrix $Q$ and the data $\{x(t_i),r(t_i)\}$ available to each agent $i\in \I$,
design a distributed data-driven algorithm to find the Lyapunov certificate $P$ in  \eqref{eqn: LE}.
\end{problem}

Next, consider a linear system with input 
\begin{equation}\label{eqn: LTI}
    \dot{x}=Ax+Bu,
\end{equation}
where $B \in \R^{n \times m}$ is the input matrix, and $(A, B)$ is stabilizable.
We aim to design an optimal controller that is able to stabilize \eqref{eqn: LTI}. In particular, we are interested in linear quadratic regulators, which are obtained using the 
solutions of the following algebraic Riccati equation (ARE) 
\begin{equation}\label{eqn: ARE}
    A^TP+PA+Q-PBR^{-1}B^TP=0,
\end{equation}
where $Q, R>0$.
Let $P^*$ denote the unique positive definite solution to the ARE. It is well-known that the controller gain given by \begin{equation}\label{eqn: Optimal state feedback}
    K^*=-R^{-1}B^TP^*,
\end{equation}
stabilizes the system and minimizes the expected LQR cost 
\[
J(K)= \mathbb{E} \left[\int_0^{\infty} x^T(t)(Q+K^TRK)x(t)dt\right],
\]
where the expectation is taken with respect to  the initial state $x_0 \in \R^n$, assumed to be zero-mean random vector satisfying $\mathbb{E}[x_0x^T_0]=I_n$.
We assume that the input matrix $B$ is known, whereas the system matrix $A$ is unknown. We discuss a partial relaxation of this assumption in Subsection \ref{subsec: Uncertainty}.

\begin{problem}\label{prb: LQR for LTI}
    Given the matrices $B, Q,R$ and the data $\{x(t_i),r(t_i), u(t_i)\}$ available to each agent $i\in \I$,
design a distributed data-driven algorithm to compute the optimal control gain $K^*$ in \eqref{eqn: Optimal state feedback}.
\end{problem}

As for the data, we impose the following assumption
\begin{assumption}\label{asm: Full row rank of X_0}   
It holds that
 $\mathrm{rank}\,X_0=n,$
 where
 \begin{align}\label{eqn: State data notation}
    \begin{split}
        X_0 &:= \begin{bmatrix}
            x(t_1) & x(t_2) & \cdots & x(t_N)
        \end{bmatrix}.
    \end{split}
\end{align}
\end{assumption}
\section{Data-based splitting scheme}\label{sec: Data distribution scheme and partial system identification}
As the first step to tackle the formulated problems, we devise a scheme which splits the unknown matrix $A$ into $N$ (low-rank) shares $A_i$, such that  $A=\sum_{i=1}^N A_i$, and each $A_i$ is available to agent $i$. \\
Assuming noiseless measurements, for system \eqref{eqn: LTIH}, we have
\begin{equation}\label{eqn: Subspace relations for LTIH}
    r(t_i)=Ax(t_i),
\end{equation}
and for \eqref{eqn: LTI} we have
\begin{equation}\label{eqn: Subspace relations}
    r(t_i)=Ax(t_i)+Bu(t_i),
\end{equation}
for each $i\in \I$.
For convenience, we continue working with the latter equation \eqref{eqn: Subspace relations} in this section, noting that \eqref{eqn: Subspace relations for LTIH} corresponds to a special case where $u(t_i)=0$, $\forall i$.
Defining 
\begin{equation}\label{eqn: Def r_hat}
    \hat{r}(t_i):=r(t_i)-Bu(t_i)
\end{equation} 
for each $i \in \I,$
observe that the system matrix $A$ satisfies
\begin{subequations}\label{eqn: Equations for estimation of A}
    \begin{equation}\label{eqn: Equation of A}
        A= \sum_{i=1}^N \hat{r}(t_i) y^T_i,    
        \end{equation}
        \text{ for any set of vectors $\{y_i\}_{i \in \I}$ such that }
    \begin{equation}\label{eqn: Right inverse of data matrix}
        I_n= \sum_{i=1}^N x(t_i)y^T_i
    \end{equation}
\end{subequations}
Note that Assumption \ref{asm: Full row rank of X_0} guarantees the existence of such $y_i$, with $i\in \I$.

The constraint \eqref{eqn: Right inverse of data matrix} can be treated as a distributed resource allocation problem, and can be solved using a distributed algorithm. 
We have provided such an algorithm in Appendix \ref{sec: Appendix A}, where the vectors/matrices in \eqref{eqn: Distributed right inverse problem statement}, can be tailored to \eqref{eqn: Right inverse of data matrix} by choosing $v(i)=x(t_i), w(i)=y_i,$ and $M=I_n$. As a result, each agent $i\in \I$ can compute a vector $y_i$ such that the resulting set of computed vectors satisfy \eqref{eqn: Right inverse of data matrix}. 
The distributed computation of the vectors $\{y_i\}_{i\in \I}$ enables each agent $i$ to compute the term 
\begin{equation}\label{eqn: Distributed equation of A}
    A_i:=\hat{r}(t_i)y^T_i,
\end{equation} 
which precisely coincides with one of the $N$ terms appearing in the summation \eqref{eqn: Equation of A}. As a result, $A=\sum_{i=1}^N A_i$, where $A_i$ given by \eqref{eqn: Distributed equation of A} is a data-dependent matrix available to agent $i$. 
\section{Distributed algorithms to find a Lyapunov certificate }\label{sec: Dist sol to LE}
Following the splitting scheme performed in Section \ref{sec: Data distribution scheme and partial system identification}, Problem \ref{prb: Lyapunov function for a stable system} reduces to 
designing a distributed algorithm to compute the unique solution $P^*$ to the Lyapunov equation 
\begin{equation}\label{eqn: Distributed LE}
    \left(\sum_{i=1}^N A^T_i\right) P + P \left(\sum_{i=1}^NA_i\right) + Q = 0,
\end{equation}
where the data-dependent matrix $A_i$ in \eqref{eqn: Distributed equation of A} is available to agent $i$, for each $i\in \I$. 
The Lyapunov equation \eqref{eqn: Distributed LE} can be solved centrally using the associated differential Lyapunov equation (DLE), i.e.:
\begin{equation}\label{eqn: DLE}
        \dot{P}=A^T P + P A + Q.
\end{equation}
It is known that solutions of the DLE globally exponentially converge to the unique equilibrium $P^*$ of the DLE  which satisfies the associated Lyapunov equation \eqref{eqn: LE}.

Next, we provide a distributed algorithm that solves Problem \ref{prb: Lyapunov function for a stable system}, with \textit{practical} convergence guarantees. 
To this end, we use the following result from the theory of blended dynamics.
\begin{lemma}\cite{jiang2021trends}\label{lem: Blended dynamics}
    Consider a networked system, whose dynamics are given by 
    \begin{equation}\label{eqn: MA gen}
        \dot{x}_i = f_i(t,x_i)+\gamma \sum_{j \in \mathcal{N}_i} (x_j-x_i),~~ i \in \I,
    \end{equation}
    where $\mathcal{N}_i$ is the set of in-neighbours of agent $i$.
    The averaged dynamics or blended dynamics are given by 
    \begin{equation}\label{eqn: Blended gen}
        \dot{s}=\frac{1}{N}\sum_{i=1}^{N}f_i(t,s).
    \end{equation} 
    Assume that the equilibrium $s^*$ is uniformly asymptotically stable for the blended dynamics \eqref{eqn: Blended gen}, and let $\mathcal{D}_b$ be an open subset of the domain of attraction of $s^*$. Define 
    \begin{equation*}\label{eqn: Domain of attraction gen}
        \mathcal{D}_x:=\{\mathds{1}_N \otimes s + w: s \in \mathcal{D}_b, w \in \R^{nN}, (\mathds{1}^T_N \otimes I_n)w=0\}.
    \end{equation*}
    Then, for any compact set $\mathcal{K} \subset \mathcal{D}_x \subset \R^{nN}$ and for any $\epsilon > 0$, there exists $\gamma^*>0$ such that, for each $\gamma>\gamma^*$ and $\col(x_1(t_0),\ldots,x_N(t_0))) \in \mathcal{K}$, the trajectories of \eqref{eqn: MA gen} satisfy 
    $\limsup_{t \to \infty} ||x_i(t)-x_j(t)|| \leq \epsilon,$ and 
         $\limsup_{t \to \infty} ||x_i(t)-s^*||\leq \epsilon,$ for all $i,j \in \mathcal{N}$.
\end{lemma}

We aim to design a dynamical algorithm in the form of \eqref{eqn: MA gen} whose  averaged dynamics \eqref{eqn: Blended gen} coincide with the vectorized form of the DLE in \eqref{eqn: DLE}. The agents are assumed to communicate over an undirected, connected graph. 
The following result addresses Problem \ref{prb: Lyapunov function for a stable system}. 

\begin{theorem}\label{thm: Dist alg for prac conv LE}
Let each agent $i\in \I$ have access to the data-dependent rank one matrix $A_i$ in \eqref{eqn: Distributed equation of A}, and employ the distributed dynamics:
\begin{equation}\label{eqn: Distributed DLE with coupling}
    \dot{P}_i=NA^T_iP_i + NP_iA_i + Q + \gamma \sum_{j \in \mathcal{N}_i} (P_j-P_i),
\end{equation}
with any arbitrary initial condition $P_i(0),~ i \in \I$.
Then, for any $\epsilon>0$, there exists a coupling gain $\gamma^*>0$ such that for every $\gamma > \gamma^*$, the solutions of the distributed dynamical algorithm \eqref{eqn: Distributed DLE with coupling} satisfy $\lim_{t \to \infty} ||P_i(t)-P^*|| \leq \epsilon$, where $P^*$ is the unique solution to \eqref{eqn: LE}. 
\end{theorem}
\BP
The averaged dynamics of \eqref{eqn: Distributed DLE with coupling} is equal to 
\begin{equation*}
    \dot{S}=\left(\sum_{i=1}^NA^T_i\right)S+S\left(\sum_{i=1}^N A_i\right)+Q,
\end{equation*}
which coincides with the differential Lyapunov equation \eqref{eqn: DLE}. Therefore, their respective vectorized versions coincide with each other. The practical convergence follows directly from Lemma \ref{lem: Blended dynamics} with the definitions $s^*=\mathrm{vec}(P^*)$, $\mathcal{D}_b=\R^{N}$, and  $\mathcal{K}=\mathcal{D}_x=\R^{nN}$. The choices of $\mathcal{D}_b$  and $\mathcal{D}_x$ are aided by the fact that $\R^n$ is the region of attraction of \eqref{eqn: DLE}.
\EP

Theorem \ref{thm: Dist alg for prac conv LE} provides a practical solution to Problem \ref{prb: Lyapunov function for a stable system}, and the convergence error $\epsilon$ can  be brought down to predefined level, by increasing the coupling gain $\gamma$.
Nevertheless, the dynamics \eqref{eqn: Distributed DLE with coupling} would never return a zero error.   
To see this, note that any synchronized solution $P_i=P_j$, $\forall i,j \in \mathcal{I}$, is inconsistent with the equilibrium of \eqref{eqn: Distributed DLE with coupling}. 
In particular, the 
Lyapunov equation 
\begin{equation}\label{eqn: Averaged DLE dynamics with synchronization}
   0= NA_i^T\bar P+ N\bar PA_i+Q
\end{equation}
does not hold for any $\bar P>0$, bearing in mind that $A_i$ has rank $1$ and thus is not Hurwitz.
Motivated by this fact, and inspired by the results from \cite{kim2023decentralized}, we augment the algorithm by an additional PI-type dynamics, which leads to 
\begin{subequations}\label{eqn: Distributed DLE with PI coupling}
        \begin{align}
        \nonumber\dot{P}_i&=NA^T_iP_i + NP_iA_i + Q + \gamma \sum_{j \in \mathcal{N}_i} (P_j-P_i) \\[-2mm]&\qquad+ \gamma \sum_{j \in \mathcal{N}_i} (Y_j-Y_i), \label{eqn: P_dot in DLE PI alg}\\
        \dot{Y}_i&=-\gamma \sum_{j \in \mathcal{N}_i} (P_j-P_i). \label{eqn: Y_dot in DLE PI alg}
        \end{align} 
    \end{subequations}
Next, we have the following result.
\begin{theorem}\label{thm: Asymptotic convergence to solution of LE}
    Let each agent $i \in \I$ have access to the data-dependent rank one matrix $A_i$ in \eqref{eqn: Distributed equation of A}, and employ the distributed dynamics \eqref{eqn: Distributed DLE with PI coupling}, with any arbitrary initial condition $(P_i(0),Y_i(0))~ i \in \I$. 
    Then, there exists a coupling gain $\gamma^*>0$ such that for any $\gamma > \gamma^*$, the solutions of the distributed dynamical algorithm \eqref{eqn: Distributed DLE with PI coupling} satisfy $\lim_{t \to \infty} P_i(t)= P^*$, where $P^*$ is the unique solution to the Lyapunov equality in \eqref{eqn: LE}.
\end{theorem}
 We require two technical results before proceeding with the proof of Theorem \ref{thm: Asymptotic convergence to solution of LE}. The first result provides a Lyapunov function certifying global exponential convergence of the DLE \eqref{eqn: DLE} to $P^*$, which will be used later.
 
 \begin{lemma}\label{lem: GES of DLE}
    Let $A$ be a Hurwitz matrix, $Q>0$, and consider  the Lyapunov differential equation  \eqref{eqn: DLE}. 
The Lyapunov function 
    \begin{equation}\label{eqn: Lyapunov function for DLE}
        V(\tilde{P}):=\mathrm{tr}({P^*}^{-1/2}\tilde{P}{P^*}^{-1}\tilde{P}{P^*}^{-1/2})
    \end{equation}
     satisfies 
     \begin{equation}\label{eqn: Dissipation rate for DLE}
         \dot{V} = -2 \mathrm{tr}({P^*}^{-1/2}\tilde{P}{P^*}^{-1}Q{P^*}^{-1}\tilde{P}{P^*}^{-1/2})
     \end{equation}
      along the trajectories of \eqref{eqn: DLE}, where $\tilde{P}:=P-P^*$.
\end{lemma}
\BP
    Since $P^*$ is the equilibrium of \eqref{eqn: DLE}, it follows that 
    \begin{equation}\label{eqn: LE inverted}
        {P^*}^{-1}A^T+A{P^*}^{-1}+{P^*}^{-1}Q{P^*}^{-1}=0.
    \end{equation}
    Using the transformation $\tilde{P}:=P-P^*,$ the DLE \eqref{eqn: DLE} can be rewritten as
    \begin{equation}\label{eqn: Shifted DLE}
        \dot{\tilde{P}}=A^T\tilde{P}+\tilde{P}A.
    \end{equation}
    Consider the Lyapunov function candidate $V(\tilde{P})$ as given in \eqref{eqn: Lyapunov function for DLE}. Then,
    \begin{align*}
        \dot{V}(\tilde{P})&=2 \mathrm{tr}((P^*)^{-1/2}\tilde{P}(P^*)^{-1} (A^T\tilde{P}+\tilde{P}A)(P^*)^{-1/2})\\
        &\overset{\eqref{eqn: Cyclic property of trace}}{=}2 \mathrm{tr} ({P^*}^{-1}\tilde{P}{P^*}^{-1}A^T \tilde{P})+ 2 \mathrm{tr} ({P^*}^{-1}\tilde{P}A{P^*}^{-1}\tilde{P} )\\
        &= 2 \mathrm{tr} ({P^*}^{-1}\tilde{P}({P^*}^{-1}A^T+A{P^*}^{-1})\tilde{P})\\
        &\overset{\eqref{eqn: LE inverted}}{=}-2 \mathrm{tr} ({P^*}^{-1}\tilde{P}{P^*}^{-1}Q{P^*}^{-1}\tilde{P})\\
        &\overset{\eqref{eqn: Cyclic property of trace}}{=} -2 \mathrm{tr}({P^*}^{-1/2}\tilde{P}{P^*}^{-1}Q{P^*}^{-1}\tilde{P}{P^*}^{-1/2}). \rlap{\qquad\;\;\,$\blacksquare$} 
    \end{align*}

Note that exponential convergence follows due to the linearity of the dynamics \eqref{eqn: DLE} and that $\dot V<0$, $\forall P\ne P^*$.

As for the second technical result, we vectorize \eqref{eqn: Distributed DLE with PI coupling} and perform algebraic manipulations to express it in an {equivalent} structured form that proves useful in the proof of Theorem~\ref{thm: Asymptotic convergence to solution of LE}. By ``equivalent,"  we mean the two dynamical systems admit a one-to-one correspondence between their solutions: every solution ${(P_i, Y_i)}_{i \in \mathcal{I}}$ to \eqref{eqn: Distributed DLE with PI coupling} can be mapped uniquely to a solution $(\xi_1, \ldots, \xi_4)$ of \eqref{eqn: Compact dynamics DLE}, and this mapping is invertible.
 
\begin{lemma}\label{lem: DLE Vectorization and simplification}
     The distributed dynamics \eqref{eqn: Distributed DLE with PI coupling} is equivalent to the linear dynamics
    \begin{equation}\label{eqn: Compact dynamics DLE}
    \begin{bmatrix}
        \dot{\xi}_1 \\ \dot{\xi}_2 \\ \dot{\xi}_3 \\ 
        \dot{\xi}_4
    \end{bmatrix}=\begin{bmatrix}
        \bar{A} & A_{12} & \mathbf{0} & \mathbf{0} \\
        A_{21} & A_{22}-\gamma A_{23} & -\gamma A_{23} & \mathbf{0}\\
        \mathbf{0} & \gamma A_{23} & \mathbf{0} & \mathbf{0}\\
        \mathbf{0} & \mathbf{0} & \mathbf{0} & \mathbf{0}
    \end{bmatrix}\begin{bmatrix}
        \xi_1 \\ \xi_2 \\ \xi_3 \\ \xi_4
    \end{bmatrix},
\end{equation}
where $\bar{A}:=(I \otimes A^T + A^T \otimes I)$, 
$\xi_1 \in \R^{n^2}$, $\xi_2 \in \R^{n^2(N-1)}$, $\xi_3 \in \R^{n^2}$, $\xi_4 \in \R^{n^2(N-1)}$, and $A_{23}>0, A_{12}$, $A_{21}$, $A_{22}$ are matrices with appropriate dimensions.
\end{lemma}
\BP
    See Appendix \ref{apx: Dist alg for LE}.
\EP 

\BPof{Theorem \ref{thm: Asymptotic convergence to solution of LE}}
As we will be working with the vectorized dynamics \eqref{eqn: Compact dynamics DLE} of the algorithm \eqref{eqn: Distributed DLE with PI coupling}, we first provide the Lyapunov function \eqref{eqn: Lyapunov function for DLE} and its dissipation rate \eqref{eqn: Dissipation rate for DLE} in terms of the vector $\tilde{p}:=\mathrm{vec}(\tilde{P})$.
To this end, note that
the vectorized form of \eqref{eqn: Shifted DLE} is given by 
\begin{equation}\label{eqn: Shifted VDLE}
    \dot{\tilde{p}}=\underbrace{(I \otimes A^T + A^T \otimes I)}_{=\bar{A}}\tilde{p},
\end{equation}
where $\tilde{p}:=\mathrm{vec}(\tilde{P})$. Using properties of trace and Kronecker product, the Lyapunov function \eqref{eqn: Lyapunov function for DLE} and its rate of change \eqref{eqn: Dissipation rate for DLE} can be written in terms of $\tilde{p}$ as
\begin{equation}\label{eqn: Lyapunov function for VDLE}
    V(\tilde{p})=\tilde{p}^T \underbrace{({P^*}^{-1} \otimes {P^*}^{-1})}_{=:\bar P}\tilde{p},\\[-2mm]
\end{equation}
and 
\begin{equation}\label{eqn: Dissipation rate for VDLE}
    \dot{V}(\tilde{p})= - 2\tilde{p}^T \underbrace{({P^*}^{-1} \otimes {P^*}^{-1}Q{P^*}^{-1})}_{=:\frac{1}{2}\bar{Q}} \tilde{p}.
\end{equation}
From \eqref{eqn: Shifted VDLE}, \eqref{eqn: Lyapunov function for VDLE}, and \eqref{eqn: Dissipation rate for VDLE}, we have
\begin{equation}\label{eqn: Lyapunov equation for VDLE}
    \bar{A}^T\bar{P}+\bar{P}\bar{A}+\bar{Q}=0.
\end{equation}
Next, we show that the origin of \eqref{eqn: Compact dynamics DLE} is asymptotically stable. Since $\dot{\xi}_4=0$ and it is not coupled with the rest of the dynamics, we omit $\xi_4$ in the following analysis. Consider the Lyapunov function candidate $V(\xi_1,\xi_2,\xi_3):=\xi^T_1\bar{P}\xi_1 + \xi^T_2\xi_2 + \xi^T_3\xi_3$, where $\bar{P}$ is defined in \eqref{eqn: Lyapunov function for VDLE}. We have
\begin{align*}
    \dot{V}&\,=\xi^T_1(\bar{P}\bar{A}+\bar{A}^T\bar{P})\xi_1+2\xi^T_1\bar{P}A_{12}\xi_2+2\xi^T_2A_{21}\xi_1\\&\,+2\xi^T_2A^T_{22}\xi_2-2\gamma \xi^T_2A_{23}\xi_2-2\gamma \xi^T_2A_{23}\xi_3+2\gamma \xi^T_3A_{23}\xi_2\\ &\overset{\eqref{eqn: Lyapunov equation for VDLE}}{=}-\xi^T_1\bar{Q}\xi_1-2\gamma \xi^T_2A_{23}\xi_2 + 2\xi^T_2A_{22}\xi_2\\&\qquad +2\xi^T_1(\bar{P}A_{12}+A^T_{21})\xi_2 \\&\,=-\begin{bmatrix}
        \xi_1\\\xi_2   
    \end{bmatrix}^T\begin{bmatrix}
        \bar{Q} & \bar{P}A_{12} + A^T_{21} \\ * & 2\gamma A_{23}-(A_{22}+A^T_{22})
    \end{bmatrix}\begin{bmatrix}
            \xi_1 \\ \xi_2
        \end{bmatrix}.
\end{align*}
Noting that $A_{23}>0$ and $\bar{Q}>0$, there exists a large enough $\gamma>0$ such that $\dot{V}<0$. By using La Salle's invariance principle, we conclude that $\xi_1, \xi_2$ and $\xi_3$ asymptotically converge to zero. 
Finally, from the proof of Lemma \ref{lem: DLE Vectorization and simplification}, it follows that  
$\xi_1=0$ and $\xi_2=0$ correspond to $P_i=P^*$ for all $i\in \I$, which completes the proof.
\EP

Note that there are two additional terms in the distributed algorithm \eqref{eqn: Distributed DLE with PI coupling}, compared to \eqref{eqn: Distributed DLE with coupling}. 
The integral law \eqref{eqn: Y_dot in DLE PI alg} ensures that there are no disagreements among the $P_i$ variables at steady-state. The added term in \eqref{eqn: P_dot in DLE PI alg} modifies the equation \eqref{eqn: Averaged DLE dynamics with synchronization} to 
\[
0=NA_i^TP^*+ NP^*A_i+\gamma\sum_{j \in \mathcal{N}_i}(Y^*_j-Y^*_i)+Q,\\[-2mm]
\]
which renders $P^*$ a feasible equilibrium, unlike the previous case. As a result, at the equilibrium of \eqref{eqn: Distributed DLE with PI coupling}, we have $P_i=P_j$, $\forall i,j\in \I$ and averaging the dynamics for all $i$ recovers the Lyapunov equation \eqref{eqn: Distributed LE}.
\begin{remark}\label{prop: Closed form gamma DLE}
    Theorem \ref{thm: Asymptotic convergence to solution of LE} is an existence result, i.e., it states that there exists a large enough $\gamma$ such that the algorithm \eqref{eqn: Distributed DLE with PI coupling} converges to the unique solution $P^*$ of
    the Lyapunov equation \eqref{eqn: LE}. 
    In particular, from the proof of Theorem \ref{thm: Asymptotic convergence to solution of LE}, any $\gamma$ satisfying the inequality 
    \begin{equation*}
        2\gamma A_{23}> A_{22}+A^T_{22} + (A^T_{12}\bar{P}+A_{21})\bar{Q}^{-1}(\bar{P}A_{12}+A^T_{21})  
    \end{equation*}
    can be chosen.  
    Noting that $A_{23}>0$, such a selection is possible if bounds on the quantities involved are available; otherwise, an iterative procedure should be used.   
\end{remark}

\section{Distributed algorithm to solve LQR problem of LTI system}\label{sec: Dist sol to ARE}
By leveraging the splitting scheme carried out in Section \ref{sec: Data distribution scheme and partial system identification}, Problem \ref{prb: LQR for LTI} reduces to 
a distributed computation of the unique stabilizing solution $P^*$ to the algebraic Riccati equation 
\begin{equation}\label{eqn: Distributed ARE}
    \left(\sum_{i=1}^N A^T_i \right)P+P\left(\sum_{i=1}^N A_i \right) + Q - P D P = 0,
\end{equation}
where $D:=BR^{-1}B^T$, and $A_i$ is the data-dependent matrix in \eqref{eqn: Distributed equation of A} available to the $i$th agent.
Once $P^*$ is computed, the optimal control policy is obtained by \eqref{eqn: Optimal state feedback}. 

The stabilizing solution to the algebraic Riccati equation \eqref{eqn: Distributed ARE} can be computed using the associated differential Riccati equation
\begin{equation}\label{eqn: DRE}
    \dot{P}=A^TP+PA+Q-PDP,
\end{equation}
where $A=\sum_{i=1}^N A_i$.
Similar to Section \ref{sec: Dist sol to LE}, we resort to blended dynamics of the form \eqref{eqn: MA gen} and aim to design a distributed algorithm such that the averaged dynamics coincides with the differential Riccati equation \eqref{eqn: DRE}.
This idea together with the property of exponential convergence of the solutions of the DRE \eqref{eqn: DRE} to the stabilizing solution, enables the design of such an algorithm as formally stated next.
\begin{theorem}\label{thm: Dist alg for prac conv ARE}
Let each agent $i \in \I$ have access to the data-dependent rank one matrix $A_i$ in \eqref{eqn: Equation of A}, and employ the distributed  dynamics: 
\begin{equation}\label{eqn: Distributed DRE with coupling}
    \dot{P}_i=NA^T_iP_i+NP_iA_i+Q-P_iDP_i+\gamma\sum_{j \in \mathcal{N}_i}(P_j-P_i),
\end{equation}
with $P_i(0)=\mathbf{0}, \forall ~ i\in\I$. Then, for any $\epsilon>0$, there exists a coupling gain $\gamma^*>0$ such that for any $\gamma>\gamma^*$, the solutions of the distributed dynamical algorithm \eqref{eqn: Distributed DRE with coupling} satisfy $\limsup_{t \to \infty} ||P_i(t)-P^*||\leq \epsilon$, where $P^*$ is the unique stabilizing solution to \eqref{eqn: ARE}.
\end{theorem}
\BP
The averaged dynamics of \eqref{eqn: Distributed DRE with coupling} is equal to 
\begin{equation*}
    \dot{S}=\left(\sum_{i=1}^NA^T_i\right)S+S\left(\sum_{i=1}^N A_i\right)+Q-SDS,
\end{equation*}
which coincides with the differential Riccati equation \eqref{eqn: DRE}. Therefore, their respective vectorized versions coincide with each other. Note that the solutions of the DRE \eqref{eqn: DRE} initiated in the positive semidefinite cone exponentially converge to the stabilizing solution $P^*$ \cite[Theorem 3]{callier1994convergence},\cite{callier1981criterion}. Hence, practical convergence follows directly from Lemma \ref{lem: Blended dynamics} with $s^*=\mathrm{vec}(P^*)$, $\mathcal{D}_b=\{Y|Y\geq \0\} \subset \R^{N}$, and  $\mathcal{K}=\mathcal{D}_x=\{\mathbf{0}\}$.
\EP

Due to arguments similar to those presented after Theorem \ref{thm: Dist alg for prac conv LE}, it is impossible to achieve convergence to the exact stabilizing solution using the dynamics \eqref{eqn: Distributed DRE with coupling}.
We again introduce additional PI-type dynamics to \eqref{eqn: Distributed DRE with coupling}, in order to guarantee asymptotic convergence to the desired solution. 
We emphasize that establishing exact convergence in the case of the DRE requires significantly more careful treatment and derivation due to the inherent nonlinearity of the Riccati dynamics. Unlike the linear DLE case, global convergence does not hold for the DRE, which necessitates the analysis of suitably constructed Lyapunov level sets.
Moreover, since the true solution $P^*$ is unknown, it may not be possible to initialize the dynamics in a sufficiently small neighborhood of $P^*$, thus making linearization-based approaches  impractical. 
The resulting algorithm and its accompanying formal guarantees are presented in the subsequent theorem, which addresses Problem~\ref{prb: LQR for LTI}. We note that the optimal controller $K^*$ can be then computed a posteriori via \eqref{eqn: Optimal state feedback}.

\begin{theorem}\label{thm: Asymptotic convergence to solution of ARE}
Let each agent $i\in\I$ have access to data-dependent rank one matrix $A_i$ in \eqref{eqn: Distributed equation of A}, and employ the distributed dynamics:  
\begin{subequations}\label{eqn: Distributed DRE with PI coupling}
    \begin{align}
        \nonumber \dot{P}_i=&NA^T_iP_i+NP_iA_i+Q-P_iDP_i\\&\quad+ \gamma\sum_{j \in \mathcal{N}_i}(P_j-P_i)+ \gamma\sum_{j \in \mathcal{N}_i}(Y_j-Y_i), \\
        \dot{Y}_i=&-\gamma\sum_{j \in \mathcal{N}_i}(P_j-P_i),
    \end{align}
\end{subequations}
with $P_i(0)=\0, ~Y_i(0)=\0, \forall ~i\in\I$. Then, there exists a coupling gain $\gamma^*>0$ such that for any $\gamma> \gamma^*$, the solutions of the  distributed dynamical algorithm \eqref{eqn: Distributed DRE with PI coupling} satisfy $\lim_{t \to \infty} P_i(t)= P^*$, where $P^*$ is the unique stabilizing solution to \eqref{eqn: ARE}.
\end{theorem}
 
Note that the origin is picked as a suitable and accessible initialization point for all agents.
In order to prove Theorem \ref{thm: Asymptotic convergence to solution of ARE}, we need a few technical results. 
The first result concerns exponential convergence of the trajectories of DRE \eqref{eqn: DRE}, initialized in the positive semidefinite cone,  
to the unique stabilizing solution.
While this property follows from explicit formulas available for $P(t)$ \cite[Theorem 3]{callier1994convergence},\cite{callier1981criterion}, we instead require a Lyapunov-based argument, which will be used in the proof of Theorem \ref{thm: Asymptotic convergence to solution of ARE}. This result is formally stated next.
\footnote{To the best of our knowledge, no nonlinear Lyapunov analysis of the DRE has been reported in the literature, making this technical lemma of independent interest.}

\begin{lemma}\label{lem: ES of DRE}
    Assume that the pair $(A,B)$ is stabilizable and $Q, R>0$. The solutions of the differential Riccati equation 
    \eqref{eqn: DRE}
    starting at $P(0)\geq 0$ exponentially converge to the stabilizing solution $P^*>0$ of the associated ARE in \eqref{eqn: ARE}. In particular, the positive definite Lyapunov function $V(\tilde{P})$ defined as 
    \begin{equation}\label{eqn: Lyapunov function for DRE}
V(\tilde{P}):=\mathrm{tr}\left(P^{*^{-1/2}}\tilde{P}P^{*^{-1}}\tilde{P}P^{*^{-1/2}}\right) 
    \end{equation}
    satisfies 
    \begin{equation}\label{eqn: Dissipation rate for DRE}
         \dot{V} \leq -2\mathrm{tr}\left(P^{*^{-1/2}}\tilde{P}P^{*^{-1}}QP^{*^{-1}}\tilde{P}P^{*^{-1/2}}\right) \leq -2{\rho}^{-1} V,
    \end{equation}
    with $\tilde{P}:=P-P^*$ and $\rho:=||{P^*}^{1/2}Q^{-1}{P^*}^{1/2}||$, for any solution of \eqref{eqn: DRE} initialized at $P(0)\geq 0$.
\end{lemma}
\BP
    To begin with, we note 
the invariance property of the positive semidefinite cone with respect to the DRE \eqref{eqn: DRE}. In particular, from \cite[Theorem 4.1.6]{abou2012matrix} it follows that
\begin{equation}\label{inv-psd}
P(0)\geq 0 \Rightarrow P(t)\geq 0, \quad \text{for all } t \geq 0.
\end{equation}
Next, we rewrite the DRE \eqref{eqn: DRE} using the coordinate transformation $\tilde{P}:=P-P^*$ as
\begin{equation}\label{eqn: Shifted DRE}
    \dot{\tilde{P}}=A^T_{cl}\tilde{P} + \tilde{P} A_{cl} - \tilde{P}D\tilde{P},
\end{equation}
where
 $   D:=BR^{-1}B^T, \quad A_{cl}:=A-DP^*.$
Writing the invariance property \eqref{inv-psd} in the new coordinates yields
\begin{equation}\label{eqn: Invariance of shifted PSD cone}
    \tilde{P}(0) \geq -P^* \implies \tilde{P}(t) \geq -P^*,~ \forall t>0.
\end{equation}
In order to simplify the analysis, we define a second transformation $Z:={P^*}^{-1/2}\tilde{P}{P^*}^{-1/2}$, and write the dynamics in $Z$ as 
\begin{equation}\label{eqn: Z dyn}
    \dot{Z}=\tilde{A}^T_{cl}Z+Z\tilde{A}_{cl}-Z\tilde{D}Z,
\end{equation}
where $\tilde{A}_{cl}:={P^*}^{1/2} A_{cl} {P^*}^{-1/2}$, and $\tilde{D}:={P^*}^{1/2} D {P^*}^{1/2}$. \\
Noting that $P^*$ is the unique stabilizing solution to the associated algebraic Riccati equation \eqref{eqn: ARE}, we have
\begin{align*}
    A^TP^*+P^*A+Q-P^*DP^*&=0,\\
    A^T_{cl}P^*+P^*A_{cl}+Q+P^*DP^*&=0.
\end{align*}
Hence, by multiplying the second equality from the left and right by 
${P^*}^{-1/2}$, we obtain
\begin{equation}\label{eqn: Zero state relation DRE}
    -\tilde{A}^T_{cl} - \tilde{A}_{cl} -\tilde{D}= {P^*}^{-1/2}Q{P^*}^{-1/2} :=\tilde{Q}.
\end{equation}
Consider the Lyapunov candidate function $V(\tilde{P})$ defined in \eqref{eqn: Lyapunov function for DLE}. In $Z$ coordinates, $V=\tr(Z^2)$ and 
\begin{align*}
    \dot{V}&=2 \mathrm{tr}(Z\dot{Z})\\
           &\overset{\eqref{eqn: Z dyn}}{=}2 \mathrm{tr}(Z\tilde{A_{cl}}^TZ+Z^2\tilde{A}_{cl}-Z^2\tilde{D}Z)\\
           &\overset{\eqref{eqn: Cyclic property of trace}}{=}2 \mathrm{tr}\left(Z(\tilde{A}^T_{cl}+\tilde{A}_{cl}-Z\tilde{D})Z\right)\\
           &\overset{\eqref{eqn: Zero state relation DRE}}{=}-2\mathrm{tr}\left(Z(\tilde{D}+Z\tilde{D}+\tilde{Q})Z\right)\\
           &=-2\mathrm{tr}(Z\tilde{Q}Z) -2\mathrm{tr}((I+Z)(Z\tilde{D}Z))\\
           &\overset{\eqref{eqn: Cyclic property of trace}}{=}-2\mathrm{tr}(Z\tilde{Q}Z) -2\mathrm{tr}((I+Z)^{1/2}Z\tilde{D}Z(I+Z)^{1/2}),
\end{align*}
where $(I+Z)^{1/2}$ is well-defined, since $Z(t)+I={P^*}^{-1/2}(\tilde{P}(t)+P^*){P^*}^{-1/2} \geq 0$ as a consequence of \eqref{eqn: Invariance of shifted PSD cone}. Hence, we conclude that
\begin{align}\label{eqn: Quadratic convergence of RDE}
\nonumber
    \dot{V}&\leq -2\mathrm{tr}(Z\tilde{Q}Z) \leq -2 \sigma_{\rm min}(\tilde{Q}) V\\
    &=-2 \sigma_{\rm min}(\tilde{Q})
    \mathrm{tr}({P^*}^{-1/2}\tilde{P}{P^*}^{-1}\tilde{P}{P^*}^{-1/2})
\end{align}
for all $\tilde{P} \geq -P^*$.
Bearing in mind the invariance property \eqref{eqn: Invariance of shifted PSD cone}, exponential convergence of $P$ to $P^*$ follows from \eqref{eqn: Quadratic convergence of RDE}. In addition, \eqref{eqn: Quadratic convergence of RDE} implies \eqref{eqn: Dissipation rate for DRE}, noting that $\sigma_{\rm min}(\tilde{Q})=1/||{\tilde {Q}^{-1}}||=\rho^{-1}$.
\EP

Next, we require the vectorized version of Lemma \ref{lem: ES of DRE}, i.e., the Lyapunov function \eqref{eqn: Lyapunov function for DRE} and its dissipation rate \eqref{eqn: Dissipation rate for DRE} in terms of the vector  $\tilde{p}:=\mathrm{vec}(\tilde{P})$.
Define the linear map $\phi^r_D:\R^{n^2}\to\R^{n^2 \times n^2}$ as 
\begin{equation}\label{eqn:phi-r}
\phi^r_D(v)=I_n \otimes (\mathrm{mat}(v)D).
\end{equation}
Since we deal with the same matrix $D$ throughout this work, we drop the suffix $D$ for brevity, and denote the function by $\phi^r(.)$. We also define
\begin{equation}\label{eqn: Abar_cl}
        \bar{A}_{cl}:=I \otimes A^T_{cl} + A^T_{cl} \otimes I.
\end{equation}
Then the vectorized dynamics of $\tilde{P}$, i.e., vectorized form of \eqref{eqn: Shifted DRE} is given by 
\begin{equation}\label{eqn: Shifted VDRE}
    \dot{\tilde{p}}=(I \otimes A^T_{cl} + A^T_{cl} \otimes I)\tilde{p} -\phi^r(\tilde{p})\tilde{p}=: (\bar{A}_{cl}-\phi^r(\tilde{p}))\tilde{p}.
\end{equation}
Using properties of trace and Kronecker product, the Lyapunov function \eqref{eqn: Lyapunov function for DRE} and its rate of change \eqref{eqn: Dissipation rate for DRE} are written in terms of $\tilde{p}$ as
\begin{equation}\label{eqn: Lyapunov function for VDRE}
    V_R(\tilde{p})=\tilde{p}^T \underbrace{(P^{*^{-1}} \otimes P^{*^{-1}})}_{=:P_R}\tilde{p},
\end{equation}
and 
\begin{equation}\label{eqn: Dissipation rate for VDRE}
    \dot{V}_R(\tilde{p})\leq -2 \tilde{p}^T \underbrace{(P^{*^{-1}} \otimes P^{*^{-1}}QP^{*^{-1}})}_{=:\frac{1}{2}Q_R} \tilde{p},
\end{equation}
for all $\tilde{p}$ such that $\mathrm{mat}(\tilde{p})\geq -P^*$.
The equations \eqref{eqn: Lyapunov function for VDRE} and \eqref{eqn: Dissipation rate for VDRE} imply that for all $x \in \{x|\mathrm{mat}(x)\geq -P^*\}$, we have
\begin{equation}\label{eqn: Lyapunov type inequality for DRE}
    (\bar{A}_{cl}-\phi^r(x))^TP_R+P_R (\bar{A}_{cl}-\phi^r(x)) \leq -Q_R.
\end{equation}

Finally, we perform some algebraic manipulations similar to those in Section \ref{sec: Dist sol to LE} and obtain equivalent dynamics to \eqref{eqn: Distributed DRE with PI coupling} in a suitable form, whose structure is instrumental in the proof of Theorem \ref{thm: Asymptotic convergence to solution of ARE}.

\begin{lemma}\label{lem: Simplified dynamics DRE}
Let $\bar A_{cl}$ be given by \eqref{eqn: Abar_cl} and $\phi^r(.)$ by \eqref{eqn:phi-r}.
The dynamics \eqref{eqn: Distributed DRE with PI coupling} is equivalent to the nonlinear dynamics
\begin{align}\label{eqn: Compact dynamics DRE}
    \begin{split}
        &\dot{\xi}=A_{\xi}\xi-g(\xi)\xi,\\
        &\dot{\xi}_4=0,
    \end{split}
\end{align}
        where $\xi:=\col(\xi_1,\xi_2,\xi_3)$,
        \begin{equation*}A_{\xi}=\begin{bmatrix}
        \bar{A}_{cl} & A_{12} & \mathbf{0} \\
        A_{21} & A_{22}-\gamma A_{23} & -\gamma A_{23}\\
        \mathbf{0} & \gamma A_{23} & \mathbf{0}
    \end{bmatrix},
    \end{equation*} 
    \begin{equation*}
    g(\xi)=\begin{bmatrix}
        -\phi^r(\xi_1) & -g_{12}(\xi_2) & \mathbf{0} \\
        \mathbf{0} & -g_{21}(\xi_1)-g_{22}(\xi_2) & \mathbf{0}\\
        \mathbf{0} & \mathbf{0} & \mathbf{0}
    \end{bmatrix},
    \end{equation*}
for some matrices $ A_{23} > 0, A_{12}, A_{21}, A_{22}$, and linear functions $g_{12}, g_{21}, g_{22}$.
Moreover, $(P_i(0), Y_i(0))=(\0,\0),~\forall i \in \I$ is mapped to 
$\xi_1(0)=-\mathrm{vec}(P^*)$, $\xi_2(0)=0$, $\xi_3(0)=-\frac{1}{\gamma}A^{-1}_{23}A_{21}\mathrm{vec}(P^*)$, and $\xi_4(0)=0$.
\end{lemma}
\BP
    See Appendix \ref{apx: Dist alg for ARE}
\EP  

\BPof{Theorem \ref{thm: Asymptotic convergence to solution of ARE}}
    We prove the claimed convergence properties on the equivalent representation \eqref{eqn: Compact dynamics DRE}. Since $\dot{\xi}_4=0$ and $\xi_4$ does not appear elsewhere, we consider only the state variables $\xi_1, \xi_2, \xi_3$ in the rest of the proof. Consider the Lyapunov function candidate
    \begin{equation}
        V(\xi):=\xi^T\begin{bmatrix}
            P_R & \0 & \0\\
            \0 & \alpha+\beta & \beta\\
            \0 & \beta & \alpha+2\beta
        \end{bmatrix}\xi,    
    \end{equation}
    for some $\alpha>0,\beta>0$ to be chosen later, where $P_R$ is defined in \eqref{eqn: Lyapunov function for VDRE}. It is easy to see that $V$ is positive definite. 
    To establish the result, we show that there exists a subset $S$ of the state space such that
    \begin{enumerate}
        \item[(P1)] $\dot{V}< 0$, $\forall \xi \in S \backslash \{0\}$,
        \item[(P2)] $S$ is invariant with respect to \eqref{eqn: Compact dynamics DRE}, and
        \item[(P3)] $\xi(0):=\big(-\mathrm{vec}(P^*),0,-\frac{1}{\gamma}A^{-1}_{23}A_{21}\mathrm{vec}(P^*)\big) \in S$.
    \end{enumerate}
Establishing (P1): We have, 
\begin{align*}
    \dot{V}=&\xi^T_1(\bar{A}_{cl}-\phi^r(\xi_1))^TP_R\xi_1+\xi^T_1P_R (\bar{A}_{cl}-\phi^r(\xi_1))\xi_1\\+&\xi^T_1P_R(A_{12}-g_{12}(\xi_2))\xi_2 + \xi^T_2(A^T_{12}-g^T_{12}(\xi_2))P_R\xi_1 \\
    +&2 (\alpha+\beta) \xi^T_2A_{21} \xi_1 -2(\alpha+\beta)\gamma\xi^T_2A_{23}\xi_3 \\+& 2 (\alpha+\beta) \xi^T_2(A_{22}-\gamma A_{23}-g_{21}(\xi_1)-g_{22}(\xi_2)) \xi_2 \\
    +&2 (\alpha+2\beta)\gamma \xi^T_3A_{23}\xi_2 + 2\beta \gamma \xi^T_2 A_{23} \xi_2+2\beta \xi^T_3 A_{21}\xi_1\\
    +&2\beta \xi^T_3(A_{22}-\gamma A_{23}-g_{21}(\xi_1)-g_{22}(\xi_2))\xi_2\\-&2\beta \gamma \xi^T_3 A_{23}\xi_3.
    \end{align*}
    Using the inequality \eqref{eqn: Lyapunov type inequality for DRE}, we have
    \begin{align*}
    \dot{V}&\leq -\xi^T_1Q_R\xi_1 - 2 \alpha \gamma \xi^T_2A_{23}\xi_2 -2\beta \gamma \xi^T_3A_{23}\xi_3 \\&+ 2\beta \xi^T_3A_{21}\xi_1 + 2\beta \xi^T_2(A^T_{22}-g^T_{21}(\xi_1)-g^T_{22}(\xi_2))\xi_3 \\&+  2(\alpha+\beta) \xi^T_2 \underbrace{(A_{22}-g_{21}(\xi_1)-g_{22}(\xi_2))}_{=:h_{22}(\xi_1,\xi_2)}\xi_2 \\&+2\xi^T_1(\overbrace{P_RA_{12}-P_Rg_{12}(\xi_2)}^{=:h_{12}(\xi_2)} +(\alpha+\beta)A^T_{21})\xi_2,
\end{align*}
for all $
\xi \in \Psi:=\{(\xi_1,\xi_2,\xi_3)|\; \mat(\xi_1) \geq -P^*\}.
$
This can be written compactly as
$\dot{V} \leq -\xi^T H(\xi_1,\xi_2) \xi$, for all  $\xi \in \Psi$, where
\[H(\xi_1,\xi_2):=\begin{bmatrix}   \begin{smallmatrix}
        Q_R & -h_{12}-(\alpha+\beta) A^T_{21} & -\beta A^T_{21}\\[1mm] 
        * & - (\alpha+\beta) (h_{22}+h^T_{22})+2 \alpha \gamma A_{23}&  -\beta h^T_{22}\\[1mm]
        * & * & 2\beta \gamma A_{23}
        \end{smallmatrix}
    \end{bmatrix}.\]
Bearing in mind that $Q_R>0$, using Schur's complement, we have $H(\xi_1,\xi_2) > 0$ if and only if 
\begin{align*}\label{eqn: Schur's complement}
    & \begin{bmatrix}
        2 \alpha \gamma A_{23} - (\alpha+\beta) (h_{22}+h^T_{22}) &  -\beta h^T_{22}\\
        -\beta h_{22} & 2\beta \gamma A_{23}
    \end{bmatrix}-\\&\begin{bmatrix}
        (h^T_{12}+(\alpha+\beta)A_{21})Q^{-1}_R(\star) & *\\ \beta A_{21}Q^{-1}_R(h_{12}+(\alpha+\beta)A^T_{21}) & \beta^2 A_{21}Q^{-1}_R A^T_{21}
    \end{bmatrix}>0,
\end{align*}
with the abridged notation $TQ_R^{-1}(\star):=TQ_R^{-1}T^T$ valid for any matrix $T$.
We intersect $\Psi$ with the Lyapunov sublevel sets $\Omega_c:=\{\xi|\; V(\xi)\leq c\}$, and define $S_c:=\Psi \cap \Omega_c$. 
For any given $\alpha,\beta, c>0$, 
there exists a $\gamma_{\alpha,\beta,c}>0$ such that $H(\xi_1,\xi_2) > 0$,  for all $\gamma>\gamma_{\alpha,\beta,c}$ and $\xi \in S_c$.
This follows from the fact that
$h_{12}$ and $h_{22}$ are affine functions and thus remain bounded on any bounded set. Moreover, we have used the fact that $\gamma$ appears only in the first matrix in the above inequality, while the second matrix is independent of $\gamma$. 
Consequently, $\dot{V} < 0$ for all $\xi \in S_c$, provided that $\gamma>\gamma_{\alpha,\beta,c}$. This establishes the first property (P1) with $S\equiv S_c$.  

Establishing (P2): Next, we proceed with establishing the invariance of $S_c$ by restricting the value of $c$. We note that since $\xi=0$ belongs to the interior of $\Psi$, the set $S_c$ becomes invariant for arbitrarily small values of $c$. However, such values  potentially exclude $\xi(0)$ from the guaranteed region of attraction, thereby violating (P3).
To circumvent this, we establish invariance of the set $S_c$ 
without restricting $c$ to be arbitrarily small. 
As a first step towards this goal, we note that the dynamics of $\xi_1$ given in \eqref{eqn: Compact dynamics DRE}, namely $\dot{\xi}_1=(\bar{A}_{cl}-\phi^r(\xi_1))\xi_1 + (A_{12}-g_{12}(\xi_2))\xi_2$, can be viewed as the dynamics \eqref{eqn: Shifted VDRE} perturbed by the term $(A_{12}-g_{12}(\xi_2))\xi_2$. Switching to matrix form, the dynamics of $\Xi_1:=\mat(\xi_1)$ given by
\begin{equation}\label{eqn: Xi1 dyn}
    \dot{\Xi}_1=A^T_{cl}\Xi_1+\Xi_1A_{cl}-\Xi_1D\Xi_1+G(\xi_2),
\end{equation}
can be interpreted as a perturbed version of \eqref{eqn: Shifted DRE} with the perturbation given by $G(\xi_2):=\mat((A_{12}-g_{12}(\xi_2))\xi_2)$. Observe that the cone $\{\tilde{P}|\tilde{P} \geq -P^*\}$ is invariant with respect to \eqref{eqn: Shifted DRE} (\cite[Theorem 4.1.6]{abou2012matrix}).
If $||G(\xi_2)|| \leq \sigma_{\rm min}(Q)$, then the cone $\{\Xi_1|\Xi_1 \geq -P^*\}$ remains invariant with respect to \eqref{eqn: Xi1 dyn}, following \cite[Theorem 4.1.6]{abou2012matrix}. It is easy to see from its definition that the matrix-valued function $G(\xi_2)$ is Lipschitz continuous in $\xi_2$. Hence, there exists a scalar $b>0$ such that $||\xi_2||\leq b \implies ||G(\xi_2)||\leq \sigma_{\rm min}(Q)$. Consequently, if $||\xi_2||\leq b,~ \forall t \geq 0$, then, the set $\{\Xi_1|\Xi_1\geq -P^*\}$ is invariant with respect to \eqref{eqn: Xi1 dyn}. Switching back to the vector form, we have that the set $\Psi \cap \Omega_{\alpha b^2}$ is invariant with respect to the dynamics \eqref{eqn: Compact dynamics DRE}, as $V(\xi_1,\xi_2,\xi_3)\leq \alpha b^2 \implies ||\xi_2||\leq b$. Note that in the latter implication, we used the fact that 
\[
\alpha ||\xi_2||^2 \leq \xi^T \Big(\blkdiag (P_R, \alpha, \alpha+\beta)\Big) \xi\leq V(\xi). 
\]
This completes establishing the second property (P2), with $S\equiv S_c$ and $c=\alpha b^2$.

Establishing (P3):
It remains to enforce that the given initial condition $\xi(0)$ belongs to the set $\Psi \cap \Omega_{\alpha b^2}$ by suitably choosing  $\alpha$, $\beta$ and $\gamma$.  
Clearly, $\xi(0)\in \Psi$, 
since $\mat(\xi_1(0))=-P^*$. 
Moreover $\xi(0)\in \Omega_{\alpha b^2}$ if and only if $V(\xi(0)) \leq \alpha b^2$, i.e.,
\[{p^*}^TP_R\,p^*+\left(\frac{\alpha+2\beta}{\gamma^2}\right) {p^*}^TA^T_{21}A^{-2}_{23}A_{21}p^* 
\leq \alpha b^2,\]
where $p^*:=\mathrm{vec}(P^*)$.
It is easy to see that the above inequality holds for any triple $(\alpha,\beta,\gamma)$ satisfying 
\begin{equation}\label{eqn: Lower bound on gamma}
\gamma^2>\frac{(\alpha+2\beta){p^*}^TA^T_{21}A^{-2}_{23}A_{21}p^*}{\alpha b^2-{p^*}^TP_Rp^*} >0. 
\end{equation}
This suggests selecting $\alpha> ({p^*}^TP_Rp^*)/b^2$. Once $\alpha$ is chosen, the above condition is satisfied for any
$\gamma > \bar \gamma_{\alpha,\beta}$, with $\bar \gamma_{\alpha,\beta}$ denoting the square root of the ratio in \eqref{eqn: Lower bound on gamma}. Combining this with the conditions resulting from properties (P1)-(P2), i.e. $\gamma >\gamma_{\alpha,\beta,c}$ with $c=\alpha b^2$, the proof is complete by setting $\gamma^*=\max(\bar \gamma_{\alpha,\beta},\, \gamma_{\alpha,\beta,c})$.
\EP

\section{Robustness}\label{sec: Robustness}
\subsection{Partially known input matrix}\label{subsec: Uncertainty}
Recall that the results in Section \ref{sec: Dist sol to ARE}  provide solutions to Problem \ref{prb: LQR for LTI}, under the assumption that the matrix $B$ is known.  
Here, we partially relax this assumption by allowing for some uncertainty in the input matrix $B$. This weakened assumption is formally stated next. 
\begin{assumption}\label{asm: Public knowledge of B with uncertainty}
    The true input matrix is given by $B=B_0+\Delta_B$, with $||\Delta_B||_2 \leq \varepsilon$, where $B_0$ and $\varepsilon$ are known but $\Delta_B$ is unknown.
\end{assumption}

The first bottleneck to extend the results to this case originates from the fact that the data-based splitting scheme in Section \ref{sec: Data distribution scheme and partial system identification}, and specifically the vector $\hat r(t_i)$ in \eqref{eqn: Def r_hat}, depend on the matrix $B$.
To address this challenge, we slightly modify the splitting scheme and strengthen Assumption \ref{asm: Full row rank of X_0} as follows.
\vspace{-5mm}
\begin{assumption}\label{asm: Full row rank of X,U} 
It holds that $\mathrm{rank}\begin{bmatrix}
    X_0\\U_0
\end{bmatrix}=n+m$, where
\vspace{-2mm}
\begin{align*}
    \begin{split}
        X_0 &:= \begin{bmatrix}
            x(t_1) & x(t_2) & \cdots & x(t_N)
        \end{bmatrix},\\
        U_0 &:= \begin{bmatrix}
            u(t_1) & u(t_2) & \cdots & u(t_N)
        \end{bmatrix}.
    \end{split}
\end{align*}
\end{assumption}

\vspace{-3mm}
This assumption routinely appears in data-driven control; see e.g. \cite{de2019formulas,dorfler2023certainty}.
Let $r(t_i)$ be defined as in \eqref{eqn: Subspace relations}. 
Then the system matrix satisfies 
\begin{subequations}\label{eqn: Equations for estimation of A ext}
    \begin{equation}\label{eqn: Equation of A ext}
            A= \sum_{i=1}^N r(t_i) y^T_i,    
        \end{equation}
        \text{ for any set of vectors $\{y_i\}_{i \in \I}$ such that }
        \begin{equation}\label{eqn: Right inverse of data matrix ext}
            I_{n \times n}= \sum_{i=1}^N x(t_i)y^T_i,
    \end{equation}
    \begin{equation}\label{eqn: Additional constraint on data matrix}        \0_{m \times n}=\sum_{i=1}^N u(t_i)y^T_i.
    \end{equation}
\end{subequations}
Note that Assumption \ref{asm: Full row rank of X,U} guarantees the existence of such $y_i$, $\forall i\in \I$.
This is analogous to \eqref{eqn: Equations for estimation of A}, with the newly added constraint \eqref{eqn: Additional constraint on data matrix}. In fact, this newly added constraint allows us to rewrite \eqref{eqn: Equation of A} as \eqref{eqn: Equation of A ext}, which does not directly depend on the matrix $B$.
As such, the distributed algorithm in Appendix \ref{sec: Appendix A} can be employed with $v(i):=\begin{bmatrix}
    x(t_i)\\u(t_i)
\end{bmatrix}, w(i):=y_i$, and $M:=\begin{bmatrix}
    I_{n \times n} \\ \0_{m \times n}
\end{bmatrix}$. Consequently, each agent $i\in \mathcal{I}$ can compute the term 
\begin{equation}\label{eqn: Distributed equation of A uncertain B}
    A_i:=r(t_i)y^T_i 
\end{equation} 
in terms of data,  and $A=\sum_{i=1}^N A_i$.

Due to the uncertainty in $B$, rather than computing the  stabilizing solution $P^*$ to \eqref{eqn: Distributed ARE}, we work with the nominal ARE
\begin{equation}\label{eqn: Distributed ARE with nominal B}
\left(\sum_{i=1}^N A^T_i \right)X+X\left(\sum_{i=1}^N A_i \right) + Q - X B_0R^{-1}B^T_0 X = 0,
\end{equation}
where the data-based matrix $A_i$ is given by \eqref{eqn: Distributed equation of A uncertain B}. 
Then, we can leverage the algorithm in Theorem \ref{thm: Asymptotic convergence to solution of ARE} with $B$ substituted by $B_0$ to find the unique stabilizing solution of \eqref{eqn: Distributed ARE with nominal B}, denoted by $P_0$, from the distributed data samples. 
The corresponding state feedback is given by $K_0:=-R^{-1}B^T_0P_0$. Next, we provide rigorous suboptimal guarantees for the controller $K_0$, inspired by the certainty-equivalence approach \cite{dorfler2023certainty}. 
\begin{proposition}\label{prop: Suboptimal uncertainty}
    Let Assumption \ref{asm: Public knowledge of B with uncertainty} hold. Denote the unique positive definite solution to \eqref{eqn: Distributed ARE with nominal B} by $P_0$, and the corresponding state feedback by $K_0:=-R^{-1}B^T_0P_0$. 
    If 
    \begin{equation}\label{eqn: Suboptimality - bound on uncertainty}
        \varepsilon < \frac{\sqrt{\sigma_{\rm min}(Q) \sigma_{\rm min}(R)}}{2 \rm{tr}(P_0)},
    \end{equation}
    then, $K_0$ is a stabilizing controller for the true system, with the corresponding LQR cost $J(K_0) \leq \frac{1}{\eta} \rm{tr}(P_0)$, where 
    \begin{equation}\label{eqn: Definition eta}
       \eta=1-2\varepsilon \frac{\tr(P_0)}{\sqrt{\sigma_{\rm min}(Q)\sigma_{\rm min}(R)}} .
    \end{equation}
\end{proposition}
\BP 
 Consider the closed-loop system obtained by applying the state feedback $K_0$ to the system in \eqref{eqn: LTI}, appended with an output  
\begin{align*}
    \begin{split}
        \dot{x}(t)&=(A+BK_0)x(t)+x_0\delta(t),\\
        y(t)&=(Q+K^T_0RK_0)^{1/2}x(t),
    \end{split}
\end{align*}
where $\delta(\cdot)$ denotes the Dirac delta function. Assuming that $A+BK_0$ is Hurwitz, the expected quadratic cost
\begin{equation*}
    \mathbb{E}\left[\int_0^{\infty} x^T(t) \left(Q+K^T_0RK_0\right)x(t) dt \right]
\end{equation*}
is equal to the squared $H_2$ norm of the transfer function from the impulse input $x_0 \delta(t)$ to the output $y(t)$. Therefore,
\begin{equation}\label{eqn: J(K_0)}
    J(K_0)=\tr(QW+K^T_0RK_0W),
\end{equation}
where $W>0$ is the unique solution to the Lyapunov equation $(A+BK_0)W+W(A+BK_0)^T+I=0$ (see \eqref{eqn: H2 norm}). 
Next, we bound the terms on the right hand side of \eqref{eqn: J(K_0)} to establish the claim of the proposition.
We first rewrite \eqref{eqn: Distributed ARE with nominal B} using $P_0$ and $K_0$ as
\begin{equation*}
    (A+B_0K_0)^TP_0+P_0(A+B_0K_0)+Q+K^T_0RK_0=0.
\end{equation*}
Then, $P_0$ coincides with the observability Gramian of the auxiliary system  
  \begin{align*}
      \begin{split}
          \dot{x}_{\rm aug}(t)&=(A+B_0K_0)x_{\rm aug}(t) + v(t),\\
          y_{\rm aug}(t)&=(Q+K^T_0RK_0)^{1/2}x_{\rm aug}(t),
      \end{split}
  \end{align*}
  and using the trace equality in \eqref{eqn: H2 norm}, we have
\begin{equation}\label{eqn: Cost equivalence ARE and Gramian}
    \tr(P_0)=\tr(QW_0+K^T_0RK_0W_0),
\end{equation}
where $W_0>0$ satisfies
\begin{equation}\label{eqn: Gramian for nominal system}
    (A+B_0K_0)W_0 + W_0(A+B_0K_0)^T + I = 0. 
\end{equation}
Consequently, we have 
\begin{align}\label{eqn: Bounds on part of sum}
    \begin{split}
        0 &\leq \rm{tr}(QW_0) \leq \rm{tr}(P_0),\\
        0 &\leq \rm{tr}(K^T_0RK_0W_0) \leq \rm{tr}(P_0).
    \end{split}
\end{align}
In addition, we have
\begin{multline*}
    \rm{tr}(QW_0) = \rm{tr}(W^{1/2}_0QW^{1/2}_0) \geq \sigma_{\rm min}(Q)\rm{tr}(W_0) \\\geq \sigma_{\rm min}(Q)||W_0||_2.
\end{multline*}
Therefore, 
\begin{equation}\label{eqn: Bound on W0}
    ||W_0||_2 \leq \frac{\rm{tr}(Q W_0)}{\sigma_{\rm min}(Q)} \leq \frac{\rm{tr}(P_0)}{\sigma_{\rm min}(Q)}.
\end{equation}
Moreover,
\begin{align}\label{eqn: Intermediate bound on KW}
    ||W_0K^T_0RK_0W_0||_2 &\geq \sigma_{\rm min}(R)||K_0W_0||^2_2,
\end{align}
    and 
    \begin{equation}\label{eqn: Bound on WKR}
        \begin{aligned}
            ||W_0K^T_0RK_0W_0||_2 &\leq ||W_0||_2||W^{\frac{1}{2}}_0K^T_0RK_0W^{\frac{1}{2}}_0||_2 \\ &\leq ||W_0||_2 \rm{tr}(W^{\frac{1}{2}}_0K^T_0RK_0W^{\frac{1}{2}}_0) \\&= ||W_0||_2 \rm{tr}(K^T_0RK_0W_0) \\&\overset{\eqref{eqn: Bound on W0},\eqref{eqn: Bounds on part of sum}}{\leq} \frac{\rm{tr}^2(P_0)}{\sigma_{\rm min}(Q)}\,.
        \end{aligned}
    \end{equation}
    From \eqref{eqn: Intermediate bound on KW} and \eqref{eqn: Bound on WKR}, it follows that
    \begin{equation}\label{eqn: Bound on K0W0}
        ||K_0W_0||_2 \leq \frac{\rm{tr}(P_0)}{\sqrt{\sigma_{\rm min}(Q) \sigma_{\rm min}(R)}}.
    \end{equation} 
    Hence,
    \begin{multline}\label{eqn: Bound on norm of  add term}
            ||\Delta_B K_0 W_0 + W_0 K^T_0 \Delta^T_B||_2 \leq 2 ||\Delta_B||_2 ||K_0W_0||_2 \\\overset{\eqref{eqn: Bound on K0W0}}{\leq} \frac{2 \varepsilon \rm{tr}(P_0)}{\sqrt{\sigma_{\rm min}(Q) \sigma_{\rm min}(R)}}.
        \end{multline}
    By defining $\eta$ as in \eqref{eqn: Definition eta}, from \eqref{eqn: Bound on norm of  add term} we have
    \begin{equation}\label{eqn: Bound on add term 2}
        \Delta_B K_0 W_0 + W_0 K^T_0 \Delta^T_B \leq (1-\eta) I.
    \end{equation}
    Adding \eqref{eqn: Bound on add term 2} and \eqref{eqn: Gramian for nominal system}, and noting $B=B_0+\Delta_B$,  gives 
    \begin{equation*}
        (A+BK_0)W_0 + W_0(A+BK_0)^T + I \leq (1-\eta)I,
    \end{equation*}
    which leads to the inequality 
    \begin{equation}\label{eqn: Final inequality}
        (A+BK_0)(\frac{1}{\eta} W_0) + (\frac{1}{\eta} W_0)(A+BK_0)^T + I < 0.
    \end{equation}
    By \eqref{eqn: Suboptimality - bound on uncertainty}, we have $\eta>0$, and thus $K_0$ is a stabilizing suboptimal state feedback for the true system.
    Noting that 
    $W\leq \frac{1}{\eta} W_0$, from \eqref{eqn: J(K_0)} we have
    \begin{align*}
        J(K_0) &= \rm{tr}(Q W+K^T_0RK_0 W)
        \\&\leq \rm{tr}\left(Q(\frac{1}{\eta} W_0)+K^T_0RK_0(\frac{1}{\eta} W_0)\right) \overset{\eqref{eqn: Cost equivalence ARE and Gramian}}{=} \frac{1}{\eta} \rm{tr}(P_0).
    \end{align*}
\EP
\vspace{-3mm}
\subsection{Noisy data}\label{subsec: Noisy data}
In this subsection, we retain the assumption that the true input matrix $B$ is known, but consider the scenario where the rate of change of state measurements are affected by noise, i.e., \eqref{eqn: Subspace relations} modifies to $r(t_i)=Ax(t_i)+Bu(t_i)+d(t_i)$, where $d(t_i),~i \in \I$, accounts for the presence of noise. Recalling the definition of $\hat{r}(t_i)$ from \eqref{eqn: Def r_hat}, the system matrix satisfies the modified relation \eqref{eqn: Equations for estimation of A}, namely,
\begin{equation}\label{eqn: Equation of A noisy rate data}
    A = \sum_{i=1}^N \left(\hat{r}(t_i)-d(t_i)\right)y^T_i,
\end{equation}
for any set of vectors $\{y_i\}_{i \in \I}$,
satisfying \eqref{eqn: Right inverse of data matrix}. Since the noise sequence $\{d(t_i)\}_{i \in \I}$ is unknown, \eqref{eqn: Equation of A noisy rate data} can be rewritten as $A=A_0 + \Delta_A$ where $A_0:=\sum_{i=1}^N A_i$, $A_i$ is given by \eqref{eqn: Distributed equation of A}, and $\Delta_A:=\sum_{i=1}^N d(t_i)y^T_i$.  As explained in Section \ref{sec: Data distribution scheme and partial system identification}, under Assumption \ref{asm: Full row rank of X_0}, each agent $i$ can compute the share $A_i$ using the algorithm detailed in Appendix \ref{sec: Appendix A}. 
On the other hand, the matrix $\Delta_A$ is unknown to all agents.  
Let 
\begin{equation}
        \Delta_d:=\begin{bmatrix}
            d(t_1) & \cdots & d(t_N)
        \end{bmatrix}, 
        Y^T_0:=\begin{bmatrix}
            y(1) & \cdots & y(N)
        \end{bmatrix}.   
\end{equation}
Then, we have $\Delta_A=\Delta_dY_0$.
By momentarily neglecting $\Delta_A$, \`a la certainty equivalence approach, the algorithm in Theorem \ref{thm: Asymptotic convergence to solution of ARE} provides us with the stabilizing solution $\bar{P}_0$ to the nominal ARE
\begin{equation}\label{eqn: ARE nominal noise}
     A^T_0 P + P A_0 + Q - P B R^{-1} B^T P = 0,
\end{equation}
which is the original ARE \eqref{eqn: ARE} with the matrix $A$ replaced by $A_0$.     
Next, we provide suboptimal guarantees for the corresponding controller gain by assuming that the noise quantities are bounded.

\begin{proposition}\label{prop: Suboptimal noisy}
    Let Assumption \ref{asm: Full row rank of X_0} hold. Let $\bar{P}_0$ be the unique positive definite solution to the nominal ARE \eqref{eqn: ARE nominal noise} and $\bar{K}_0:=-R^{-1}B^T\bar{P}_0$ be
    the corresponding state feedback gain. Assume that the noise admits the energy bound $||\Delta_d||_2 \leq \tau$, for some $\tau \geq 0$.
    If 
    \begin{equation}\label{eqn: Suboptimality - bound on noise}
        \tau < \frac{\sigma_{\rm min}(Q) \sigma_{\rm min}(X_0)}{2\rm{tr}(\bar{P}_0)},
    \end{equation}
    then, $\bar{K}_0$ is a stabilizing controller for the true system, with the corresponding LQR cost $J(\bar{K}_0) \leq \zeta \rm{tr}(\bar{P}_0)$, where 
    \begin{equation}\label{eqn: Definition zeta}
        \zeta:= \frac{\sigma_{\rm min}(Q) \sigma_{\rm min}(X_0)}{\sigma_{\rm min}(Q)\sigma_{\rm min}(X_0) - 2 \tau \rm{tr}(\bar{P}_0)}. 
    \end{equation}
\end{proposition}
\BP
We can mimic the same arguments that led to \eqref{eqn: Cost equivalence ARE and Gramian} in the proof of Proposition \ref{prop: Suboptimal uncertainty} to arrive at the following equality 
\begin{equation}\label{eqn: Cost equivalence ARE and Gramian noisy}
    \mathrm{tr}(\bar{P}_0) = \mathrm{tr}(QW_0+\bar{K}^T_0R\bar{K}_0W_0),
\end{equation}
where $W_0>0$ satisfies 
\begin{equation}\label{eqn: Gramian for nominal system noisy}
    (A_0+B\bar{K}_0)\bar{W}_0 + \bar{W}_0(A_0+B\bar{K}_0)^T + I = 0. 
\end{equation}
 By \eqref{eqn: Cost equivalence ARE and Gramian noisy}, and following a  similar reasoning that led to \eqref{eqn: Bound on W0}, we obtain 
    \begin{equation}\label{eqn: Bound on W_hat}
        ||\bar{W}_0||_2 \leq \frac{\rm{tr}(\bar{P}_0)}{\sigma_{\rm min}(Q)}.
    \end{equation}
In addition, a bound on the error matrix $\Delta_A:=\Delta_dY_0$ is computed as
\begin{equation}\label{eqn: Bound on Del A using SNR}
    ||\Delta_A||_2 =||\Delta_dY_0||_2 \overset{\eqref{eqn: Right inverse of data matrix}}{\leq} \frac{\tau}{\sigma_{\rm min}(X_0)}.
\end{equation}
Combining \eqref{eqn: Bound on W_hat} and \eqref{eqn: Bound on Del A using SNR}, we have 
\begin{equation*}
    ||\Delta_A \bar{W}_0 + \bar{W}_0 \Delta^T_A||_2 \leq \frac{2 \rm{tr}(\bar{P}_0) \tau}{\sigma_{\rm min}(Q) \sigma_{\rm min}(X_0)}.
\end{equation*}
Therefore, we have 
 $  \Delta_A \bar{W}_0 + \bar{W}_0 \Delta^T_A \leq (1-\zeta^{-1})I.$
The rest of the proof is analogous to the treatment following \eqref{eqn: Bound on add term 2} in the proof of Proposition \ref{prop: Suboptimal uncertainty}.
\EP

\section{Case studies}\label{sec: Numerical simulation}
In the following subsections, we apply the proposed algorithms to two practical systems to validate the theoretical results and assess their performance under representative operating conditions.\footnote{Due to space constraints, the numerical values of large data and system matrices are provided at \url{https://github.com/Surya-Malladi/Sim_data_DDDC}.}

\subsection{Distributed data-driven computation of Lyapunov function: Quadruple tank}\label{subsec: Quad tank}
We illustrate the results of Section \ref{sec: Dist sol to LE}, by numerically computing a Lyapunov function for a stable system. We consider the 4-state linearized quadruple tank model from \cite{johansson2002quadruple} with zero input, and set $Q=I$. The data generated from the system is spread across $N=4$ agents, with one sample available to each agent. The shares computed by each agent, following the scheme in Section \ref{sec: Data distribution scheme and partial system identification}, are used to implement the algorithm \eqref{eqn: Distributed DLE with PI coupling}.
The evolution of the relative error is shown in Figure \ref{fig:Quad_tank},
for the first agent estimate $P_1$; the remaining agents exhibit nearly identical trajectories.
The figure indicates that the distributed dynamics in \eqref{eqn: Distributed DLE with PI coupling} 
is able to asymptomatically compute $P^*$.  
\begin{figure}[ht]
    \centering
    \includegraphics[width=0.8\linewidth]{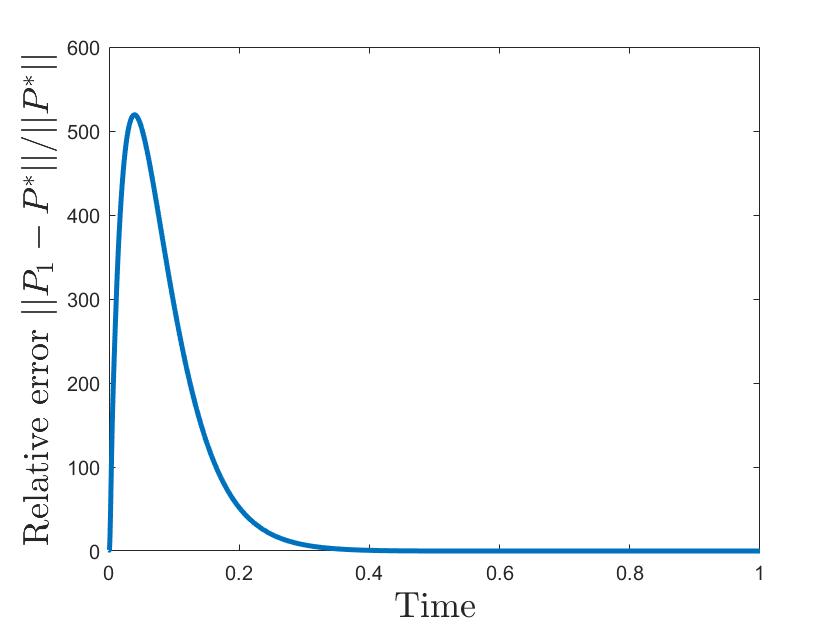}
    \caption{Quadruple tank: Asymptotic convergence to the solution of the Lyapunov equation for $\gamma=10^3$.}
    \label{fig:Quad_tank}
\end{figure}

\subsection{Distributed data-driven LQR: Helicopter hover dynamics}\label{subsec: Helicopter}
We illustrate the proposed algorithm in Section \ref{sec: Dist sol to ARE} by applying it to a linearized model of helicopter hover dynamics. The numerical values of $A$ and $B$ are obtained from \cite{padfield2008helicopter}. In addition, we set $Q=I$ and $R=I$.

First, noiseless data is generated from the system excited by randomly chosen input signals, and the data is fragmented across $N=16$ agents. The agents compute their shares of $A_i$ by running the distributed algorithm in Section \ref{sec: Appendix A}.
Figure \ref{fig:helicopter_prac_conv} shows practical convergence of the algorithm \eqref{eqn: Distributed DRE with coupling} for several choices of $\gamma$.
As can be seen, a higher $\gamma$ results in a more accurate solution to the Riccati equation.

Next, Figure \ref{fig:helicopter_asym_conv} shows the advantage of the algorithm \eqref{eqn: Distributed DRE with PI coupling}, which enables exact asymptotic computation of the true solution $P^*$ of the Riccati equation. Three arbitrarily selected datasets are shown, with the first agent’s estimate $P_1$ as representative behavior.   

\begin{figure}[ht]
    \centering
    \includegraphics[width=0.8\linewidth]{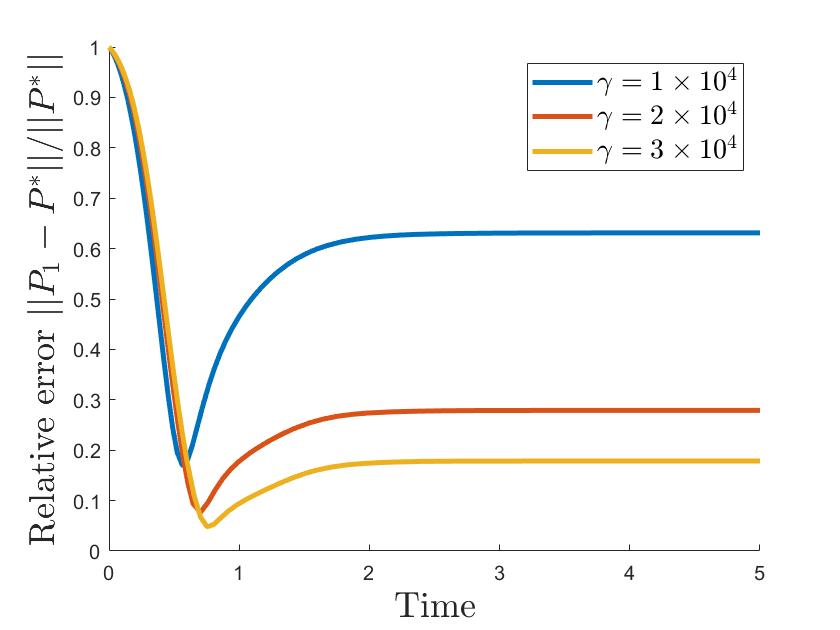}
    \caption{Helicopter dynamics: Practical convergence to the solution of the Riccati equation for different values of $\gamma$}
    \label{fig:helicopter_prac_conv}
\end{figure}
\begin{figure}[ht]
    \centering
    \includegraphics[width=0.8\linewidth]{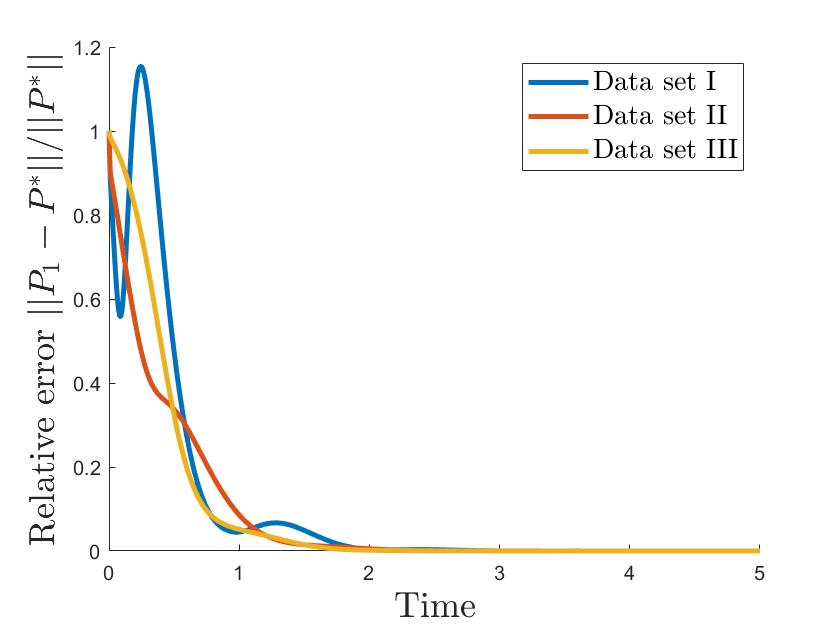}
    \caption{Helicopter dynamics: Asymptotic convergence to the solution of the Riccati equation for $\gamma=500$.}
    \label{fig:helicopter_asym_conv}
\end{figure}

We now assume only partial knowledge of the input matrix $B$. We set $\Delta_B=\kappa \Delta$, where the matrix $\Delta$, with $||\Delta||=1$, is randomly generated, and the scalar $\kappa$ specifies the strength of the uncertainty.
We apply the procedure in Subsection \ref{subsec: Uncertainty} for various levels of $\kappa$, and plot the resulting LQR cost in Figure \ref{fig: helicopter_robustness} (left), which shows how the performance of the suboptimal controller changes with the norm of the uncertainty.  

Finally, consistent with Subsection \ref{subsec: Noisy data}, we assume that the data concerning the rate of change of states are disturbed by various levels of noise, i.e. $||\Delta_d||\leq \tau$.  
We apply the procedure in Subsection \ref{subsec: Noisy data} for various levels of $\tau$, and plot the resulting LQR cost in Figure \ref{fig: helicopter_robustness} (right), which shows how the performance of the suboptimal controller varies with respect to different noise energy levels. 

\begin{figure}[ht]
    \centering
    \begin{subfigure}{0.49\linewidth}
        \centering
        \includegraphics[width=1.\textwidth]{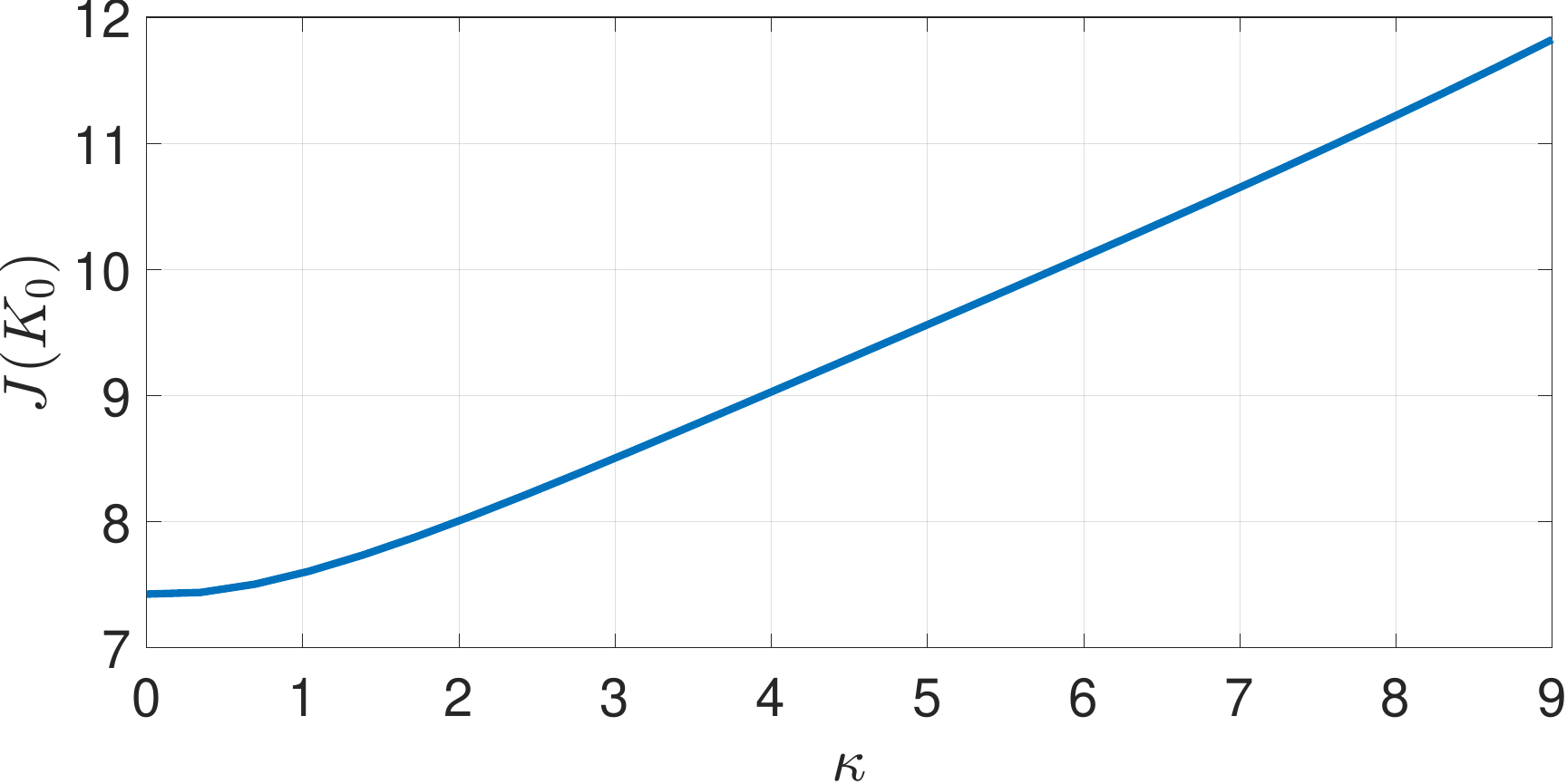}
    \end{subfigure}\hfill
    \begin{subfigure}{0.49\linewidth}
        \centering
        \includegraphics[width=\linewidth]{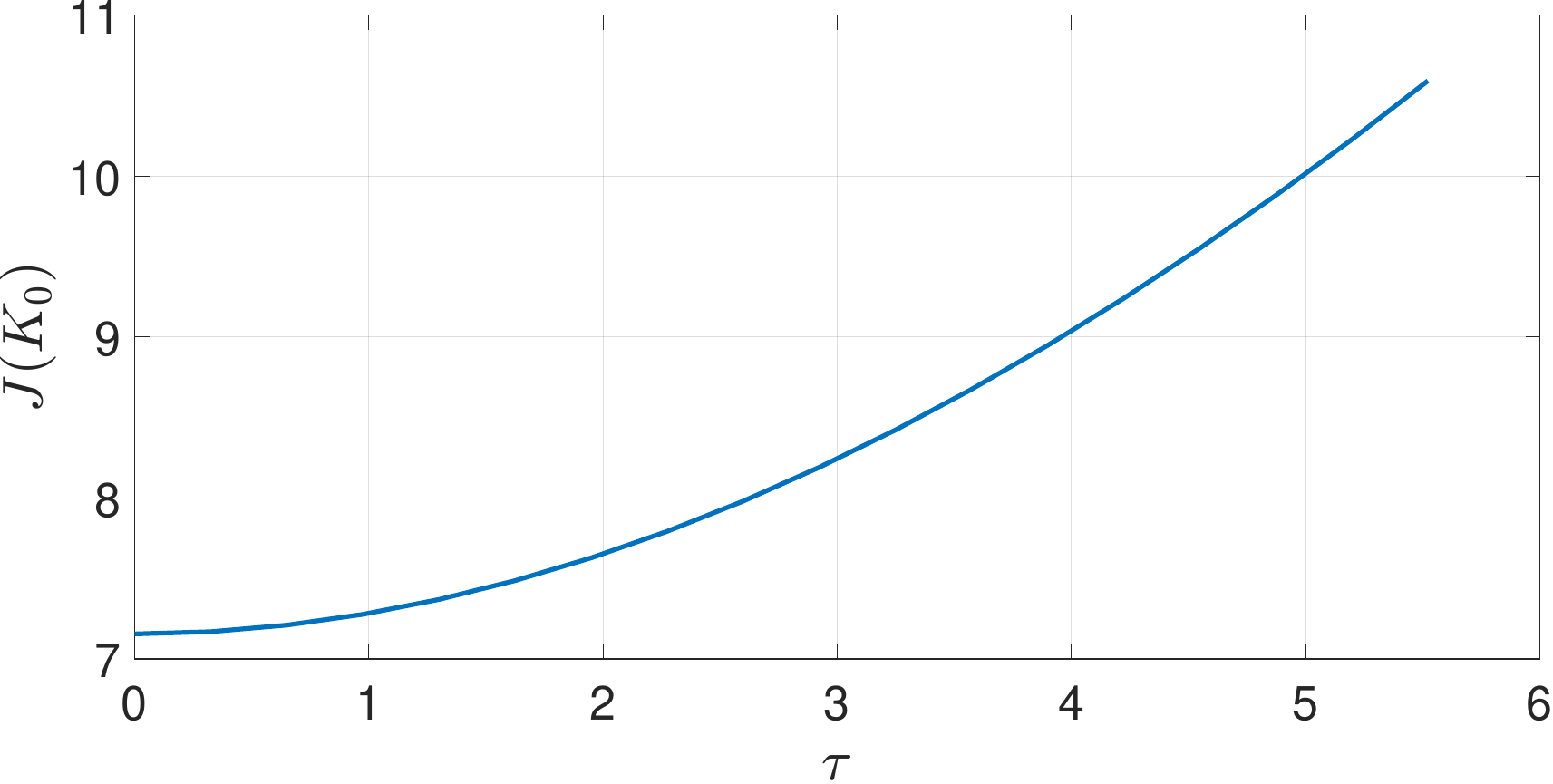}
    \end{subfigure} 
    \caption{Helicopter dynamics: Expected LQR cost vs. size of the uncertainty (left) and vs. noise energy level (right).}
    \label{fig: helicopter_robustness}
\end{figure}

\section{Conclusions}\label{sec:concl}
We investigated a data-driven control setting in which measurements are not centrally accessible but are instead held across multiple agents, each possessing only limited local samples and communicating only with neighbors. We developed distributed dynamical algorithms that enable agents to collectively compute global system quantities from such locally held data. In particular, we obtained distributed solutions to the Lyapunov equation, yielding a quadratic stability certificate for stable linear time-invariant systems, and to the algebraic Riccati equation, yielding the linear–quadratic regulator (LQR) controller. The proposed schemes achieved practical convergence, while PI-augmented variants ensured exact convergence. The resulting controller was shown to be robust to uncertainty and measurement noise. Future work will consider extensions of the framework to broader system classes and information structures.

\bibliographystyle{unsrt}
\bibliography{References}           

@book{jiang2021trends,
  title={Trends in Nonlinear and Adaptive Control: A Tribute to Laurent Praly for His 65th Birthday},
  author={Jiang, Zhong-Ping and Prieur, Christophe and Astolfi, Alessandro},
  volume={488},
  year={2021},
  publisher={Springer Nature}
}

@article{callier1981criterion,
  title={Criterion for the convergence of the solution of the {Riccati} differential equation},
  author={Callier, F and Willems, J},
  journal={IEEE Transactions on Automatic Control},
  volume={26},
  number={6},
  pages={1232--1242},
  year={1981},
  publisher={IEEE}
}

@article{callier1994convergence,
  title={Convergence of the time-invariant {Riccati} differential equation and {LQ-problem}: mechanisms of attraction},
  author={Callier, Frank M and Winkin, Joseph and Willems, Jacques L},
  journal={International journal of control},
  volume={59},
  number={4},
  pages={983--1000},
  year={1994},
  publisher={Taylor \& Francis}
}

@article{dorfler2023certainty,
  title={On the certainty-equivalence approach to direct data-driven {LQR} design},
  author={D{\"o}rfler, Florian and Tesi, Pietro and De Persis, Claudio},
  journal={IEEE Transactions on Automatic Control},
  volume={68},
  number={12},
  pages={7989--7996},
  year={2023},
  publisher={IEEE}
}

@article{kim2023decentralized,
  title={Decentralized design and plug-and-play distributed control for linear multichannel systems},
  author={Kim, Taekyoo and Lee, Donggil and Shim, Hyungbo},
  journal={IEEE Transactions on Automatic Control},
  volume={69},
  number={5},
  pages={2807--2822},
  year={2023},
  publisher={IEEE}
}

@book{abou2012matrix,
  title={Matrix Riccati equations in control and systems theory},
  author={Abou-Kandil, Hisham and Freiling, Gerhard and Ionescu, Vlad and Jank, Gerhard},
  year={2012},
  publisher={Birkh{\"a}user}
}

@inproceedings{coulson2019data,
  title={Data-enabled predictive control: In the shallows of the {DeePC}},
  author={Coulson, Jeremy and Lygeros, John and D{\"o}rfler, Florian},
  booktitle={2019 18th European control conference (ECC)},
  pages={307--312},
  year={2019},
  organization={IEEE}
}

@article{de2019formulas,
  title={Formulas for data-driven control: Stabilization, optimality, and robustness},
  author={De Persis, Claudio and Tesi, Pietro},
  journal={IEEE Transactions on Automatic Control},
  volume={65},
  number={3},
  pages={909--924},
  year={2019},
  publisher={IEEE}
}

@article{van2020data,
  title={Data informativity: A new perspective on data-driven analysis and control},
  author={Van Waarde, Henk J and Eising, Jaap and Trentelman, Harry L and Camlibel, M Kanat},
  journal={IEEE Transactions on Automatic Control},
  volume={65},
  number={11},
  pages={4753--4768},
  year={2020},
  publisher={IEEE}
}

@article{berberich2020data,
  title={Data-driven model predictive control with stability and robustness guarantees},
  author={Berberich, Julian and K{\"o}hler, Johannes and M{\"u}ller, Matthias A and Allg{\"o}wer, Frank},
  journal={IEEE transactions on automatic control},
  volume={66},
  number={4},
  pages={1702--1717},
  year={2020},
  publisher={IEEE}
}

@article{markovsky2021behavioral,
  title={Behavioral systems theory in data-driven analysis, signal processing, and control},
  author={Markovsky, Ivan and D{\"o}rfler, Florian},
  journal={Annual Reviews in Control},
  volume={52},
  pages={42--64},
  year={2021},
  publisher={Elsevier}
}

@article{campi2002virtual,
  title={Virtual reference feedback tuning: a direct method for the design of feedback controllers},
  author={Campi, Marco C and Lecchini, Andrea and Savaresi, Sergio M},
  journal={Automatica},
  volume={38},
  number={8},
  pages={1337--1346},
  year={2002},
  publisher={Elsevier}
}

@inproceedings{dai2023data,
  title={Data-driven control of nonlinear systems from input-output data},
  author={Dai, Xiaoyan and De Persis, Claudio and Monshizadeh, Nima and Tesi, Pietro},
  booktitle={2023 62nd IEEE Conference on Decision and Control (CDC)},
  pages={1613--1618},
  year={2023},
  organization={IEEE}
}

@book{khalil1996robust,
  title={Robust and optimal control},
  author={Khalil, IS and Doyle, JC and Glover, K},
  volume={2},
  year={1996},
  publisher={Prentice hall New York}
}

@article{baghcheband2024machine,
  title={Machine {Learning} {Data} {Market} {Based} on {Multiagent} {Systems}},
  author={Baghcheband, Hajar and Soares, Carlos and Reis, Luis Paulo},
  journal={IEEE Internet Computing},
  volume={28},
  number={4},
  pages={7--13},
  year={2024},
  publisher={IEEE}
}

@inproceedings{eising2022informativity,
  title={Informativity for centralized design of distributed controllers for networked systems},
  author={Eising, Jaap and Cort{\'e}s, Jorge},
  booktitle={2022 European Control Conference (ECC)},
  pages={681--686},
  year={2022},
  organization={IEEE}
}

@book{padfield2008helicopter,
  title={Helicopter flight dynamics: the theory and application of flying qualities and simulation modelling},
  author={Padfield, Gareth D},
  year={2008},
  publisher={John Wiley \& Sons}
}

@article{teixeira2015secure,
  title={A secure control framework for resource-limited adversaries},
  author={Teixeira, Andr{\'e} and Shames, Iman and Sandberg, Henrik and Johansson, Karl Henrik},
  journal={Automatica},
  volume={51},
  pages={135--148},
  year={2015},
  publisher={Elsevier}
}

@article{pasqualetti2013attack,
  title={Attack detection and identification in cyber-physical systems},
  author={Pasqualetti, Fabio and D{\"o}rfler, Florian and Bullo, Francesco},
  journal={IEEE transactions on automatic control},
  volume={58},
  number={11},
  pages={2715--2729},
  year={2013},
  publisher={IEEE}
}

@article{alonso2022data,
  title={Data-driven distributed and localized model predictive control},
  author={Alonso, Carmen Amo and Yang, Fengjun and Matni, Nikolai},
  journal={IEEE Open Journal of Control Systems},
  volume={1},
  pages={29--40},
  year={2022},
  publisher={IEEE}
}

@article{celi2023distributed,
  title={Distributed data-driven control of network systems},
  author={Celi, Federico and Baggio, Giacomo and Pasqualetti, Fabio},
  journal={IEEE Open Journal of Control Systems},
  volume={2},
  pages={93--107},
  year={2023},
  publisher={IEEE}
}

@article{lee2020tool,
  title={A tool for analysis and synthesis of heterogeneous multi-agent systems under rank-deficient coupling},
  author={Lee, Jin Gyu and Shim, Hyungbo},
  journal={Automatica},
  volume={117},
  pages={108952},
  year={2020},
  publisher={Elsevier}
}

@article{bianchin2023online,
  title={Online stochastic optimization for unknown linear systems: Data-driven controller synthesis and analysis},
  author={Bianchin, Gianluca and Vaquero, Miguel and Cortes, Jorge and Dall'Anese, Emiliano},
  journal={IEEE Transactions on Automatic Control},
  volume={69},
  number={7},
  pages={4411--4426},
  year={2023},
  publisher={IEEE}
}

@inproceedings{zheng2022fl,
  title={{FL-market}: Trading private models in federated learning},
  author={Zheng, Shuyuan and Cao, Yang and Yoshikawa, Masatoshi and Li, Huizhong and Yan, Qiang},
  booktitle={2022 IEEE International Conference on Big Data (Big Data)},
  pages={1525--1534},
  year={2022},
  organization={IEEE}
}

@article{gao2023data,
  title={Data-driven injection attack strategy for linear cyber-physical systems: An input-output data-based approach},
  author={Gao, Sheng and Zhang, Hao and Wang, Zhuping and Huang, Chao and Yan, Huaicheng},
  journal={IEEE Transactions on Network Science and Engineering},
  volume={10},
  number={6},
  pages={4082--4095},
  year={2023},
  publisher={IEEE}
}

@article{li2019optimal,
  title={Optimal stealthy innovation-based attacks with historical data in cyber-physical systems},
  author={Li, Yi-Gang and Yang, Guang-Hong},
  journal={IEEE Transactions on Systems, Man, and Cybernetics: Systems},
  volume={51},
  number={6},
  pages={3401--3411},
  year={2019},
  publisher={IEEE}
}

@article{nortmann2023direct,
  title={Direct data-driven control of linear time-varying systems},
  author={Nortmann, Benita and Mylvaganam, Thulasi},
  journal={IEEE Transactions on Automatic Control},
  volume={68},
  number={8},
  pages={4888--4895},
  year={2023},
  publisher={IEEE}
}

@article{cappello2021distributed,
  title={Distributed differential games for control of multi-agent systems},
  author={Cappello, Domenico and Mylvaganam, Thulasi},
  journal={IEEE Transactions on Control of Network Systems},
  volume={9},
  number={2},
  pages={635--646},
  year={2021},
  publisher={IEEE}
}

@inproceedings{zhao2023data,
  title={Data-enabled policy optimization for the linear quadratic regulator},
  author={Zhao, Feiran and D{\"o}rfler, Florian and You, Keyou},
  booktitle={2023 62nd IEEE Conference on Decision and Control (CDC)},
  pages={6160--6165},
  year={2023},
  organization={IEEE}
}

@article{hu2023toward,
  title={Toward a theoretical foundation of policy optimization for learning control policies},
  author={Hu, Bin and Zhang, Kaiqing and Li, Na and Mesbahi, Mehran and Fazel, Maryam and Ba{\c{s}}ar, Tamer},
  journal={Annual Review of Control, Robotics, and Autonomous Systems},
  volume={6},
  number={1},
  pages={123--158},
  year={2023},
  publisher={Annual Reviews}
}

@article{johansson2002quadruple,
  title={The quadruple-tank process: A multivariable laboratory process with an adjustable zero},
  author={Johansson, Karl Henrik},
  journal={IEEE Transactions on control systems technology},
  volume={8},
  number={3},
  pages={456--465},
  year={2002},
  publisher={IEEE}
}

\section{Appendix A: Distributed share computation}\label{sec: Appendix A}
Consider the constraint 
\begin{equation}\label{eqn: Distributed right inverse problem statement}
    \sum_{i=1}^N v(i) w^T(i) = M,
\end{equation}
where each agent $i$ has access to the vector $v(i) \in \R^{n_1}$. The matrix $M=[m_{jk}]\in\R^{n_1 \times n_2}$ is known to all agents. The objective is that each agent $i$ compute $w(i) \in \R^{n_2}$, such that the resulting computed vectors  satisfy the above constraint.  
This is equivalent to the $n_1n_2$ scalar constraints
\begin{equation}\label{eqn: Distributed right inverse as set of scalar constraints}
    \sum_{i=1}^{N}v_j(i)w_k(i)= m_{jk}, \text{ for } j=1,\ldots,n_1,k=1,\ldots,n_2,
\end{equation}
where $v_j(i)$ and $w_k(i)$ are the $j$th and $k$th elements of the vectors $v(i)$ and $w(i)$, respectively. 

Assume that $L$ is the Laplacian matrix of a connected, undirected communication graph. Then, for each pair $(j,k)$, $j=1,\ldots,n_1$ and $k=1,\ldots,n_2$, the constraints in \eqref{eqn: Distributed right inverse as set of scalar constraints} can be rewritten as 
$
    (L \mu_{jk})_i = v_j(i)w_k(i)-\frac{1}{N}m_{jk}, i=1,\ldots,N, 
$
where $\mu_{jk} \in \R^N$. By viewing the latter constraints as a distributed feasibility problem, we define the Lagrangian 
\begin{align}\label{eqn: Lagrangian for the distributed computation of right inverse}
        \nonumber\mathcal{L} =& \sum_{j=1}^{n_1} \sum_{k=1}^{n_2} \sum_{i=1}^{N} \lambda_{jk,i}\left((L\mu_{jk})_i - v_j(i)w_k(i) + \frac{1}{N}m_{jk} \right) \\ &+ \frac{k_w}{2}\sum_{i=1}^N||w(i)||^2_2,
\end{align}
where the last term of $\mathcal{L}$, with $k_w>0$, is added to ensure asymptotic stability of the subsequent  primal-dual algorithm.
Let $\mu_{jk,i}$ and $\lambda_{jk,i}$ 
denote the $i${th} element of $\mu_{jk}$ and $\lambda_{jk}$, respectively. Then, the distributed primal-dual algorithm 
\begin{align}\label{eqn: Distributed algorithm for right inverse}
    \begin{split}
        \dot{w}_k(i)&=-k_w w_k(i)+\sum_{j=1}^{n_1} v_j(i)\lambda_{jk,i},\\
        \dot{\mu}_{jk,i}&=-(L\lambda_{jk})_i,\\
        \dot{\lambda}_{jk,i}&=(L\mu_{jk})_i-v_j(i)w_k(i)+\frac{1}{N}m_{jk}
    \end{split}
\end{align} 
computes the vector $w(i)$ satisfying \eqref{eqn: Distributed right inverse problem statement} as desired. 

\section{Appendix B: Proofs of technical lemmas}\label{sec: Appendix B}
\subsection{Proof of Lemma \ref{lem: DLE Vectorization and simplification}}\label{apx: Dist alg for LE}   
Let $q:=\mathrm{vec}(Q)$, and define 
\begin{align}\label{eqn: Vec definitions pyqA}
    \begin{split}
        p_i&:=\mathrm{vec}(P_i), ~~~~~~ y_i:=\mathrm{vec}(Y_i),\\
        \bar{A}_i&:=(I_N\otimes A^T_i) + (A^T_i\otimes I_N),
    \end{split}
\end{align}
for each $i \in \I$. Then the dynamics \eqref{eqn: Distributed DLE with PI coupling} take the following vectorized form:
\begin{align*}
    \dot{p}_i&=N\bar{A}_ip_i + q + \gamma\sum_{j\in \mathcal{N}_i} (p_j-p_i) + \gamma\sum_{j\in \mathcal{N}_i} (y_j-y_i),\\
    \dot{y}_i&= -\gamma\sum_{j\in \mathcal{N}_i} (p_j-p_i).
\end{align*}
    Let $L$ denote the Laplacian matrix of the communication graph.  
    Aided by the notations 
    \begin{equation*}
        \begin{aligned}
            \mathbf{p}&:=\mathrm{col}(p_1,\ldots,p_N), &\mathbf{y}&:=\mathrm{col}(y_1,\ldots,y_N),\\
            \mathds{1}_N&:=\mathrm{col}(\underbrace{1,\ldots,1}_{N \text{ times}}),&\mathcal{A}&:=N\,\mathrm{blkdiag}(\bar{A}_1,\ldots,\bar{A}_N),\\
            \mathbf{q}&:=\mathds{1}_N\otimes q,& \mathcal{L}&:=L\otimes I_{n^2}, 
        \end{aligned}
    \end{equation*}
    the dynamics are compactly written as
    \begin{align*}
        \begin{split}
            \dot{\mathbf{p}}&=\mathcal{A}\mathbf{p}+\mathbf{q}-\gamma\mathcal{L}\mathbf{p}-\gamma\mathcal{L}\mathbf{y}, \\
            \dot{\mathbf{y}}&=\gamma\mathcal{L}\mathbf{p}.
        \end{split}
    \end{align*}
The Laplacian matrix $L$ can be decomposed as 
    $L=U\Gamma U^T$ where $\Gamma>0$ is a diagonal matrix containing the positive eigenvalues of $L$ and $U\in \R^{N \times (N-1)}$ satisfies $U^TU=I_{N-1}$ and $\mathds{1}^T U=0$. Define
\begin{equation}\label{eqn: Definitions Kronecker}
    \omega:=\mathds{1}_N \otimes I_{n^2}, \quad
        \mathcal{U}:=U \otimes I_{n^2}, \quad
        \Lambda:=\Gamma \otimes I_{n^2},
\end{equation}
and the coordinate transformation 
\begin{equation}\label{eqn: Coord trans split}
    \begin{bmatrix}
        \bar{\mathbf{p}}\\\tilde{\mathbf{p}}
    \end{bmatrix}:=\begin{bmatrix}
        \frac{1}{N}\omega^T\\\mathcal{U}^T
    \end{bmatrix}\mathbf{p}, \quad \begin{bmatrix}
        \bar{\mathbf{y}}\\\tilde{\mathbf{y}}
    \end{bmatrix}:=\begin{bmatrix}
        \frac{1}{N}\omega^T\\\mathcal{U}^T
    \end{bmatrix}\mathbf{y}.
\end{equation}
The inverse transformation of \eqref{eqn: Coord trans split} is 
\begin{equation}\label{eqn: Inverse coord trans split}
    \mathbf{p}=\omega\bar{\mathbf{p}}+\mathcal{U}\tilde{\mathbf{p}}, \quad \mathbf{y}=\omega\bar{\mathbf{y}}+\mathcal{U}\tilde{\mathbf{y}}.
\end{equation}
Notice that the state $\bar{\mathbf{p}}$ represents the average of the vectors $\{p_i\}_{i\in\I}$. Moreover, the state $\tilde{\mathbf{p}}$ is related to the errors of the vectors $p_i$ with respect to their average, and $\tilde{\mathbf{p}}=0$ if and only if $p_i=\bar{\mathbf{p}},~ \forall i\in\I$.\\
These relations follow from the properties of the Laplacian matrix and the Kronecker product:
\begin{equation}\label{eqn: Laplacian Properties Kronecker}
    \begin{aligned}
        \mathcal{L}\omega &= 0,
         & \omega^T\mathcal{U} &=0,\\
        \mathcal{U}^T\mathcal{U}&=I_{(N-1)n^2},
         & \mathcal{U}^T\mathcal{L} &=\Lambda \mathcal{U}^T,
    \end{aligned}
\end{equation}
\vspace{-2mm}
\begin{equation}\label{eqn: Property Kronecker}
    \omega^T \mathcal{A}\omega = N\sum_{i=1}^N \bar{A}_i =: N \bar{A}.
\end{equation}
The dynamics of $\mathbf{y}$ in the new coordinates reduce to 
\begin{equation}\label{eqn: Vectorized split dynamics for y}
    \dot{\bar{\mathbf{y}}}\overset{\eqref{eqn: Laplacian Properties Kronecker}}{=}0, \quad \dot{\tilde{\mathbf{y}}}\overset{\eqref{eqn: Laplacian Properties Kronecker}}{=}\gamma \Lambda \tilde{\mathbf{p}},
\end{equation}
and the dynamics of $\mathbf{p}$ in the new coordinates are written as
\begin{align*}
    \dot{\bar{\mathbf{p}}}&=\frac{1}{N}\omega^T\mathcal{A}\mathbf{p}+\frac{1}{N}\omega^T\mathbf{q} -\frac{1}{N}\gamma \omega^T \mathcal{L}(\mathbf{p}+\mathbf{y})\\
    &\overset{\eqref{eqn: Inverse coord trans split},\eqref{eqn: Laplacian Properties Kronecker}}{=}\frac{1}{N}\omega^T\mathcal{A}\omega \bar{\mathbf{p}}+ \frac{1}{N}\omega^T\mathcal{A}\mathcal{U}\tilde{\mathbf{p}} + \frac{1}{N}\omega^T (\mathds{1}_N \otimes q)\\
    &\overset{\eqref{eqn: Property Kronecker}}{=}\bar{A}\bar{\mathbf{p}} + \frac{1}{N}\omega^T\mathcal{A}\mathcal{U}\tilde{\mathbf{p}} + q.
\end{align*}
Similarly,
\begin{align*}
    \dot{\tilde{\mathbf{p}}}&=\underbrace{\mathcal{U}^T\mathcal{A}\mathbf{p}+\mathcal{U}^T\mathbf{q}}_{[\mathbf{I}]}-\underbrace{(\gamma\mathcal{U}^T\mathcal{L}\mathbf{p}+\gamma\mathcal{U}^T\mathcal{L}\mathbf{y})}_{[\mathbf{II}]}.
\end{align*}
By the use of \eqref{eqn: Inverse coord trans split} and \eqref{eqn: Definitions Kronecker}, the first expression can be written as
\begin{align*}
    [\mathbf{I}]&=\mathcal{U}^T\mathcal{A}(\omega \bar{\mathbf{p}}+\mathcal{U}\tilde{\mathbf{p}})+(U^T\otimes I_{n^2})(\mathds{1}_N\otimes q)\\
    &\overset{\eqref{eqn: Laplacian Properties Kronecker}}{=}\mathcal{U}^T\mathcal{A}\omega \bar{\mathbf{p}} + \mathcal{U}^T\mathcal{A}\mathcal{U}\tilde{\mathbf{p}},
\end{align*}
and the second one as
\begin{align*}
    [\mathbf{II}]&\overset{\eqref{eqn: Laplacian Properties Kronecker},\eqref{eqn: Inverse coord trans split}}{=}\gamma \Lambda \mathcal{U}^T (\omega \bar{\mathbf{p}}+\mathcal{U}\tilde{\mathbf{p}} + \omega \bar{\mathbf{y}}+\mathcal{U}\tilde{\mathbf{y}})\\
    \overset{\eqref{eqn: Laplacian Properties Kronecker}}{=}& \gamma \Lambda I_{(N-1)n^2} (\tilde{\mathbf{p}}+\tilde{\mathbf{y}}).\\
    \dot{\tilde{\mathbf{p}}}=&\left( \mathcal{U}^T\mathcal{A}\mathcal{U} -\gamma \Lambda \right)\tilde{\mathbf{p}} + \mathcal{U}^T\mathcal{A}\omega \bar{\mathbf{p}} -\gamma \Lambda \tilde{\mathbf{y}}.
\end{align*}
Therefore,
\begin{subequations}\label{eqn: Vectorized split dynamics for p}
    \begin{align}
        \dot{\bar{\mathbf{p}}}&=\bar{A}\bar{\mathbf{p}} + \frac{1}{N}\omega^T\mathcal{A}\mathcal{U}\tilde{\mathbf{p}} + q, \\
        \dot{\tilde{\mathbf{p}}}&=( \mathcal{U}^T\mathcal{A}\mathcal{U} -\gamma \Lambda )\tilde{\mathbf{p}} + \mathcal{U}^T\mathcal{A}\omega \bar{\mathbf{p}} -\gamma \Lambda \tilde{\mathbf{y}}.
    \end{align}
\end{subequations}
The equilibrium for the combined dynamics \eqref{eqn: Vectorized split dynamics for y} and \eqref{eqn: Vectorized split dynamics for p} is given by
\begin{equation}\label{eqn: Equilibrium of split dynamics}
    \tilde{\mathbf{p}}^*=0, \quad \bar{A}\bar{\mathbf{p}}^* + q =0, \quad \tilde{\mathbf{y}}^*=\frac{1}{\gamma}(\Lambda)^{-1}\mathcal{U}^T\mathcal{A}\omega\bar{\mathbf{p}}^*.
\end{equation}

We shift the equilibrium to the origin by defining
\begin{equation}\label{eqn: xi def}
(\xi_1, \xi_2, \xi_3, \xi_4):=(\bar{\mathbf{p}}-\bar{\mathbf{p}}^*, \tilde{\mathbf{p}}, \tilde{\mathbf{y}}-\tilde{\mathbf{y}}^*, \bar{\mathbf{y}})
\end{equation}
which transforms \eqref{eqn: Vectorized split dynamics for p} and \eqref{eqn: Vectorized split dynamics for y} into 
\begin{align*}
    \dot{\xi}_1&=\bar{A}\bar{\mathbf{p}}^*+\bar{A}\xi_1 + q + \frac{1}{N}\omega^T\mathcal{A}\mathcal{U}\xi_2
    \\&\overset{\eqref{eqn: Equilibrium of split dynamics}}{=}\bar{A} \xi_1 + \frac{1}{N}\omega^T\mathcal{A}\mathcal{U}\xi_2,\\
    \dot{\xi}_2&=( \mathcal{U}^T\mathcal{A}\mathcal{U} -\gamma \Lambda )\xi_2 + \mathcal{U}^T\mathcal{A}\omega (\bar{\mathbf{p}}^*+\xi_1) -\gamma \Lambda (\tilde{\mathbf{y}}^*+\xi_3)\\
    &\overset{\eqref{eqn: Equilibrium of split dynamics}}{=} (\mathcal{U}^T\mathcal{A}\mathcal{U} -\gamma \Lambda )\xi_2 +\mathcal{U}^T\mathcal{A}\omega \xi_1 -\gamma\Lambda \xi_3,\\
    \dot{\xi}_3&=\gamma \Lambda \xi_2.
\end{align*}
Note that all the coordinate transformations performed are reversible, and hence, there is an equivalence between the original coordinates $(P_i,Y_i)_{i \in \I}$ and $(\xi_1,\xi_2,\xi_3,\xi_4)$.\\ 
In a compact form, the dynamics are written as \eqref{eqn: Compact dynamics DLE}
where
\begin{equation*}
    \begin{aligned}
          A_{12}&:=\frac{1}{N}\omega^T\mathcal{A}\mathcal{U}, & A_{21} &:= \mathcal{U}^T\mathcal{A}\omega, \\A_{22}&:=\mathcal{U}^T\mathcal{A}\mathcal{U},& A_{23}&:=\Lambda>0. \rlap{ \quad \qquad\;\;\,$\blacksquare$}
    \end{aligned}
\end{equation*}

\subsection{Proof of Lemma \ref{lem: Simplified dynamics DRE}}\label{apx: Dist alg for ARE}
Let $D=BR^{-1}B^T$ as before. 
Recall the definition of the map $\phi^r$ from Section \ref{sec: Dist sol to ARE}, namely, $\phi^r(v)=I_n \otimes \mathrm{mat}(v)D$.  In addition, we define a map $\phi^{\ell}:\R^{n^2}\to \R^{n^2 \times n^2}$ as $\phi^{\ell}(v)=(D \,\mathrm{mat}(v))^T\otimes I_n$. The maps $\phi^k,$ $k\in \{\ell,r\}$, satisfy the following properties.
\begin{equation}\label{eqn: Linearity of phi}
    \phi^k(a_1v_1+a_2v_2)=a_1\phi^k(v_1)+a_2\phi^k(v_2), \forall v_1, v_2 \in \R^{n^2},
\end{equation}
\begin{equation}\label{eqn: Interchangeability of phi}
        \mathrm{vec}(V_1DV_2)=\phi^r(v_1)v_2=\phi^{\ell}(v_2)v_1, \forall V_1,V_2 \in \R^{n\times n}.
    \end{equation}
    We reuse the definitions of $q, p_i, y_i, \bar{A}_i, \mathbf{p}, \mathbf{y}, \mathcal{A}, \mathbf{q},\mathcal{L}$, $\mathcal{U}, \Lambda$, and $\omega$ from Appendix \ref{apx: Dist alg for LE}.
    The vectorized dynamics \eqref{eqn: Distributed DRE with PI coupling} are given as follows. For each $i \in \I$, we have
    \begin{subequations}\label{eqn: Vectorized Dist PI DRE}
        \begin{align}
            \nonumber\dot{p}_i=&(\bar{A}_i - \phi^r(p_i))p_i + q + \gamma\sum_{j\in \mathcal{N}_i} (p_j-p_i) \\&+ \gamma\sum_{j\in \mathcal{N}_i} (y_j-y_i),\\
            \dot{y}_i=& -\gamma\sum_{j\in \mathcal{N}_i} (p_j-p_i),
        \end{align}
    \end{subequations}
To write the above dynamics compactly, let $\mathbf{v}:=\mathrm{col}(v_1,\ldots,v_N)$, and $\Phi^k(\mathbf{v}):=\mathrm{diag}(\phi^k(v_1),\ldots,\phi^k(v_N))$, for $k\in \{\ell,r\}$. 
Then, \eqref{eqn: Vectorized Dist PI DRE} can be written compactly as
\begin{align}\label{eqn: Vectorized PI dynamics}
    \begin{split}
        \dot{\mathbf{p}}&=\mathcal{A}\mathbf{p}-\Phi^r(\mathbf{p})\mathbf{p}+\mathbf{q}-\gamma\mathcal{L}\mathbf{p}-\gamma\mathcal{L}\mathbf{y}, \\
        \dot{\mathbf{y}}&=\gamma\mathcal{L}\mathbf{p}.
    \end{split}
\end{align}
The maps $\Phi(k)$, $k\in \{\ell, r\}$ satisfy the following properties:
\begin{align}                   
    \Phi^k(a_1\mathbf{v_1}+a_2\mathbf{v_2})&=a_1\Phi^k(\mathbf{v_1})+a_2\Phi^k(\mathbf{v_2}),~k \in \{\ell,r\},\label{eqn: Linearity of Phi}\\
    \Phi^{\ell}(\mathbf{v_1})\mathbf{v_2}&=\Phi^r(\mathbf{v_2})\mathbf{v_1},\label{eqn: Interchangeability of Phi}\\
    \Phi^k(\omega v)&= I_N \otimes \phi^k(v).\label{eqn: Properties Kronecker 2}
\end{align}
for any $v, v_1, v_2 \in \R^{n^2}$
The first two properties follows from the properties \eqref{eqn: Linearity of phi} and \eqref{eqn: Interchangeability of phi}, while the third one results from 
the definitions of $\omega$ and the properties of Kronecker product.

By applying the same coordinate transformation as in \eqref{eqn: Coord trans split},  we obtain 
\begin{equation}\label{eqn: Vectorized split dynamics for y DRE}
    \dot{\bar{\mathbf{y}}}\overset{\eqref{eqn: Laplacian Properties Kronecker}}{=}0, \quad \dot{\tilde{\mathbf{y}}}\overset{\eqref{eqn: Laplacian Properties Kronecker}}{=}\gamma \Lambda \tilde{\mathbf{p}},
\end{equation}
and  
\begin{align*}
    \dot{\bar{\mathbf{p}}}=&\underbrace{\frac{1}{N}\left(\omega^T\mathcal{A}\mathbf{p}+\omega^T\mathbf{q}-\gamma \omega^T \mathcal{L}(\mathbf{p}+\mathbf{y})\right)}_{[\mathbf{I}]}-\underbrace{\frac{1}{N}\omega^T\Phi^r(\mathbf{p})\mathbf{p}}_{[\mathbf{II}]}
\end{align*}
Noting that $\mathbf{[I]}$ coincides with  the expression of $\dot{\bar{\mathbf{p}}}$ appeared in the proof of Lemma \ref{lem: DLE Vectorization and simplification}, we work out only the terms in $[\mathbf{II}]$. 
\begin{align*}
    &[\mathbf{II}]\overset{\eqref{eqn: Inverse coord trans split}}{=} \frac{1}{N} \omega^T \Phi^r(\omega\bar{\mathbf{p}}+\mathcal{U}\tilde{\mathbf{p}})(\omega\bar{\mathbf{p}}+\mathcal{U}\tilde{\mathbf{p}})\\
    &\overset{\eqref{eqn: Linearity of Phi}}{=}\frac{1}{N} \omega^T\bigl(\Phi^r(\omega \bar{\mathbf{p}})\omega \bar{\mathbf{p}}+\Phi^r(\omega \bar{\mathbf{p}})\mathcal{U}\tilde{\mathbf{p}}+ \Phi^r(\mathcal{U}\tilde{\mathbf{p}})\omega \bar{\mathbf{p}}  \\ & \qquad+ \Phi^r(\mathcal{U}\tilde{\mathbf{p}})\mathcal{U}\tilde{\mathbf{p}}\bigl).
\end{align*}
By using \eqref{eqn: Definitions Kronecker},\eqref{eqn: Laplacian Properties Kronecker}, \eqref{eqn: Properties Kronecker 2}, and \eqref{eqn: Interchangeability of Phi}, we obtain
\[
[\mathbf{II}]=
\phi^r(\bar{\mathbf{p}})\bar{\mathbf{p}}+\frac{1}{N}\omega^T (\Phi^r(\mathcal{U}\tilde{\mathbf{p}})\mathcal{U}\tilde{\mathbf{p}}).
\]
Similarly,
\begin{align*}
    \dot{\tilde{\mathbf{p}}}&=\underbrace{\mathcal{U}^T\mathcal{A}\mathbf{p}+\mathcal{U}^T\mathbf{q}-(\gamma\mathcal{U}^T\mathcal{L}\mathbf{p}+\gamma\mathcal{U}^T\mathcal{L}\mathbf{y})}_{[\mathbf{III}]}-\underbrace{\mathcal{U}^T\Phi^r(\mathbf{p})\mathbf{p}}_{[\mathbf{IV}]}.
    \end{align*}
Noting again that $[\mathbf{III}]$ coincides with the expression of $\dot{\tilde{\mathbf{p}}}$ obtained in the proof of Lemma \ref{lem: DLE Vectorization and simplification},
we here focus on the terms in $[\mathbf{IV}]$.
    \begin{align*}
    &[\mathbf{IV}]\overset{\eqref{eqn: Inverse coord trans split}}{=}\mathcal{U}^T\Phi^r(\omega \bar{\mathbf{p}}+\mathcal{U}\tilde{\mathbf{p}})(\omega \bar{\mathbf{p}}+\mathcal{U}\tilde{\mathbf{p}})\\
    &\overset{\eqref{eqn: Linearity of Phi}}{=}\mathcal{U}^T(\Phi^r(\omega \bar{\mathbf{p}})\omega \bar{\mathbf{p}}+\Phi^r(\omega \bar{\mathbf{p}})\mathcal{U}\tilde{\mathbf{p}}+ \Phi^r(\mathcal{U}\tilde{\mathbf{p}})\omega \bar{\mathbf{p}}\\& \qquad + \Phi^r(\mathcal{U}\tilde{\mathbf{p}})\mathcal{U}\tilde{\mathbf{p}}).
\end{align*}
Using \eqref{eqn: Definitions Kronecker}, \eqref{eqn: Laplacian Properties Kronecker}, \eqref{eqn: Properties Kronecker 2}, and \eqref{eqn: Interchangeability of Phi}, and omitting routine intermediate steps for brevity, we obtain 
\begin{equation*}
[\mathbf{IV}]=\mathcal{U}^T\Phi^r(\mathcal{U}\tilde{\mathbf{p}})\mathcal{U}\tilde{\mathbf{p}}+ \left(I_{N-1} \otimes (\phi^r(\bar{\mathbf{p}})+\phi^{\ell}(\bar{\mathbf{p}}))\right)\tilde{\mathbf{p}}
\end{equation*} 
Therefore, we have
\begin{subequations}\label{eqn: Vectorized split dynamics for p DRE}
    \begin{align}
        \nonumber\dot{\bar{\mathbf{p}}}=&\bar{A}\bar{\mathbf{p}} + \frac{1}{N}\omega^T\mathcal{A}\mathcal{U}\tilde{\mathbf{p}}{N} + q -\phi^r(\bar{\mathbf{p}})\bar{\mathbf{p}}\\&-\frac{1}{N}\omega^T (\Phi^r(\mathcal{U}\tilde{\mathbf{p}})\mathcal{U}\tilde{\mathbf{p}}),\label{eqn: DRE p_bar dyn}\\
        \nonumber\dot{\tilde{\mathbf{p}}}=&( \mathcal{U}^T\mathcal{A}\mathcal{U} -\gamma \Lambda )\tilde{\mathbf{p}} + \mathcal{U}^T\mathcal{A}\omega \bar{\mathbf{p}} -\gamma \Lambda \tilde{\mathbf{y}} - \mathcal{U}^T\Phi^r(\mathcal{U}\tilde{\mathbf{p}})\mathcal{U}\tilde{\mathbf{p}} \\&- \left(I_{N-1} \otimes (\phi^r(\bar{\mathbf{p}})+\phi^{\ell}(\bar{\mathbf{p}}))\right)\tilde{\mathbf{p}}\label{eqn: DRE p_tilde dyn}.
    \end{align}
\end{subequations}
The equilibrium for the combined dynamics \eqref{eqn: Vectorized split dynamics for y DRE} and \eqref{eqn: Vectorized split dynamics for p DRE} is given by
\begin{align}\label{eqn: Equilibrium of split dynamics DRE}
    \begin{split}
        \tilde{\mathbf{p}}^*&=0,
        \bar{A}\bar{\mathbf{p}}^* + q -\phi^r(\bar{\mathbf{p}}^*)\bar{\mathbf{p}}^*=0,\\
        \tilde{\mathbf{y}}^*&=\frac{1}{\gamma}(\Lambda)^{-1}\mathcal{U}^T\mathcal{A}\omega\bar{\mathbf{p}}^*.
    \end{split}
\end{align}
By comparing \eqref{eqn: Equilibrium of split dynamics DRE} with the vectorized algebraic Riccati 
equation \eqref{eqn: ARE}, namely $\bar{A}p-\phi^r(p)p+q=0$ , $p:=\mathrm{vec}(P)$, we find that 
$\bar{\mathbf{p}}^*=\mathrm{vec}(P^*)$ where $P^*$ is the solution to the ARE \eqref{eqn: ARE}. 
We shift the equilibrium to the origin by defining
$(\xi_1, \xi_2, \xi_3, \xi_4):=(\bar{\mathbf{p}}-\bar{\mathbf{p}}^*, \tilde{\mathbf{p}}, \tilde{\mathbf{y}}-\tilde{\mathbf{y}}^*, \bar{\mathbf{y}})$.
Consequently,  \eqref{eqn: DRE p_bar dyn} in the shifted coordinates becomes
\begin{align*}
    &\dot{\xi}_1\overset{\eqref{eqn: Linearity of phi}}{=}\bar{A}\bar{\mathbf{p}}^*+\bar{A}\xi_1 -\phi^r(\bar{\mathbf{p}}^*)\bar{\mathbf{p}}^* - \phi^r(\bar{\mathbf{p}}^*)\xi_1 - \phi^r(\xi_1)\bar{\mathbf{p}}^* \\&-\phi^r(\xi_1)\xi_1 + \frac{1}{N}\omega^T\mathcal{A}\mathcal{U}\xi_2-\frac{1}{N}\omega^T (\Phi^r(\mathcal{U}\xi_2)\mathcal{U}\xi_2) + q.
\end{align*}
Using \eqref{eqn: Equilibrium of split dynamics DRE} and \eqref{eqn: Interchangeability of phi}, the expression simplifies to 
\begin{align*}
    \dot{\xi}_1=~~&\bar{A} \xi_1- \phi^r(\bar{\mathbf{p}}^*)\xi_1 - \phi^{\ell}(\bar{\mathbf{p}}^*)\xi_1 - \phi^r(\xi_1)\xi_1\\&\:+ \frac{1}{N}\omega^T\mathcal{A}\mathcal{U}\xi_2-\frac{1}{N}\omega^T (\Phi^r(\mathcal{U}\xi_2)\mathcal{U}\xi_2)\\
    =&\bar{A}_{cl}\xi_1 -\phi^r(\xi_1)\xi_1+ \frac{1}{N}\omega^T\mathcal{A}\mathcal{U}\xi_2\\&\:-\frac{1}{N}\omega^T (\Phi^r(\mathcal{U}\xi_2)\mathcal{U}\xi_2),
\end{align*}
where the second equality follows from the fact that $\bar{A}_{cl}=\bar{A}-\phi^r(\bar{\mathbf{p}}^*)-\phi^{\ell}(\bar{\mathbf{p}}^*)$, bearing in mind the definitions of $\phi^k$ and $\bar{A}_{cl}$ in \eqref{eqn: Abar_cl}.
Similarly,  \eqref{eqn: DRE p_tilde dyn} in the shifted coordinates takes the form 
\begin{align*}
    \dot{\xi}_2\overset{\eqref{eqn: Linearity of phi}}{=}&( \mathcal{U}^T\mathcal{A}\mathcal{U} -\gamma \Lambda )\xi_2 + \mathcal{U}^T\mathcal{A}\omega \bar{\mathbf{p}}^*+\mathcal{U}^T\mathcal{A}\omega \xi_1 \\&-\gamma \Lambda \tilde{\mathbf{y}}^* -\gamma \Lambda \xi_3 - \mathcal{U}^T\Phi^r(\mathcal{U}\xi_2)\mathcal{U}\xi_2 \\&- \left(I_{N-1} \otimes (\phi^r(\bar{\mathbf{p}}^*)+\phi^{\ell}(\bar{\mathbf{p}}^*))\right)\xi_2 \\&- \left(I_{N-1} \otimes (\phi^r(\xi_1)+\phi^{\ell}(\xi_1))\right)\xi_2.
\end{align*}
By \eqref{eqn: Equilibrium of split dynamics DRE} and \eqref{eqn: Laplacian Properties Kronecker}, this simplifies to
\begin{align*}
    \dot{\xi}_2&=(\mathcal{U}^T\mathcal{A}\mathcal{U} -\gamma \Lambda )\xi_2 +\mathcal{U}^T\mathcal{A}\omega \xi_1 -\gamma\Lambda \xi_3 \\&\quad-\left((U^T\otimes I_{n^2})(I_N \otimes (\phi^r(\bar{\mathbf{p}}^*)+\phi^{\ell}(\bar{\mathbf{p}}^*)))(U \otimes I_{n^2})\right)\\&\quad-\left(I_{N-1} \otimes (\phi^r(\xi_1)+\phi^{\ell}(\xi_1))\right)\xi_2- \mathcal{U}^T\Phi^r(\mathcal{U}\xi_2)\mathcal{U}\xi_2, 
\end{align*}
and subsequently to
\begin{align*}
    \dot \xi_2&=(\mathcal{U}^T(\mathcal{A}-I_N \otimes (\phi^r(\bar{\mathbf{p}}^*)+\phi^{\ell}(\bar{\mathbf{p}}^*)))\mathcal{U} -\gamma \Lambda )\xi_2\\&\quad +\mathcal{U}^T\mathcal{A}\omega \xi_1 -\gamma\Lambda \xi_3 -\mathcal{U}^T\Phi^r(\mathcal{U}\xi_2)\mathcal{U}\xi_2\\&\quad -\left(I_{N-1} \otimes (\phi^r(\xi_1)+\phi^{\ell}(\xi_1))\right)\xi_2\\
    &=(\mathcal{U}^T\mathcal{A}_{cl}\mathcal{U}-\gamma\Lambda)\xi_2+\mathcal{U}^T\mathcal{A}\omega \xi_1 -\gamma\Lambda \xi_3\\&\quad -\mathcal{U}^T\Phi^r(\mathcal{U}\xi_2)\mathcal{U}\xi_2-\left(I_{N-1} \otimes (\phi^r(\xi_1)+\phi^{\ell}(\xi_1))\right)\xi_2,
    \end{align*}
    where $\mathcal{A}_{cl}:=\mathrm{diag}(\bar{A}_{1,cl},\ldots,\bar{A}_{N,cl})$, $\bar{A}_{i,cl}:=I_n \otimes A^T_{i,cl}+ A^T_{i,cl}\otimes I_n$, and $A_{i,cl}:=A_i-DP^*$. Consequently, the dynamics of $\xi_1$, $\xi_2$, and $\xi_3$ can be written in a compact form as \eqref{eqn: Compact dynamics DRE},
where 
\begin{align*}
    &A_{12}:=\frac{1}{N}\omega^T\mathcal{A}\mathcal{U}, A_{21}:= \mathcal{U}^T\mathcal{A}\omega, A_{22}:=\mathcal{U}^T\mathcal{A}_{cl}\mathcal{U},\\ &A_{23}:=\Lambda, g_{21}(\xi_1):= I_{N-1} \otimes (\phi^r(\xi_1)+ \phi^{\ell}(\xi_1)),\\
    &g_{12}(\xi_2):= \frac{1}{N}\omega^T\Phi^r(\mathcal{U}\xi_2)\mathcal{U},\; g_{22}(\xi_2):= \mathcal{U}^T \Phi^r(\mathcal{U}\xi_2)\mathcal{U}. \rlap{$\blacksquare$}
\end{align*}
\end{document}